\documentclass[a4paper, 10pt, conference, onecolumn]{ieeeconf}      

\IEEEoverridecommandlockouts                              
\usepackage[dvipdfmx]{graphicx}

\usepackage{comment} 
\usepackage{color} 
\usepackage{bm} 

\usepackage{amsmath,amssymb}
\usepackage{algorithm,algorithmic}
\usepackage{amsfonts}
\usepackage{physics}
\allowdisplaybreaks[4]
\usepackage{braket}
\usepackage{url}
\usepackage{lmodern}

\newcommand{\mcal}[1]{\mathcal{#1}}
\newcommand{\mb}[1]{\mathbb{#1}}

\newcommand{\pl}{\partial}
\newcommand{\ve}{\varepsilon}

\newtheorem{theo}{Theorem}
\newtheorem{prop}{Proposition}
\newtheorem{lemm}{Lemma}

\newcommand{\argmin}{\mathop{\rm argmin}\limits}

\title{
\LARGE \bf
Pontryagin's Minimum Principle and Forward-Backward Sweep Method for the System of HJB-FP Equations in Memory-Limited Partially Observable Stochastic Control
}

\author{Takehiro Tottori$^{1}$ and Tetsuya J. Kobayashi$^{1,2,3,4}$
\thanks{$^{1}$Department of Mathematical Informatics, Graduate School of Information Science and Technology, The University of Tokyo, Tokyo 113-8654, Japan}
\thanks{$^{2}$Institute of Industrial Science, The University of Tokyo, Tokyo 153-8505, Japan}
\thanks{$^{3}$Department of Electrical Engineering and Information Systems, Graduate School of Engineering, The University of Tokyo, Tokyo 113-8654, Japan}
\thanks{$^{4}$Universal Biology Institute, The University of Tokyo, Tokyo 113-8654, Japan}
}

\begin{document}

\maketitle
\thispagestyle{empty}
\pagestyle{empty}

\begin{abstract}
Memory-limited partially observable stochastic control (ML-POSC) is the stochastic optimal control problem under incomplete information and memory limitation. In order to obtain the optimal control function of ML-POSC, a system of the forward Fokker-Planck (FP) equation and the backward Hamilton-Jacobi-Bellman (HJB) equation needs to be solved. In this work, we firstly show that the system of HJB-FP equations can be interpreted via the Pontryagin's minimum principle on the probability density function space. Based on this interpretation, we then propose the forward-backward sweep method (FBSM) to ML-POSC, which has been used in the Pontryagin's minimum principle. FBSM is an algorithm to compute the forward FP equation and the backward HJB equation alternately. Although the convergence of FBSM is generally not guaranteed, it is guaranteed in ML-POSC because the coupling of HJB-FP equations is limited to the optimal control function in ML-POSC. 
\end{abstract}

\section{Introduction}
In many practical applications of the stochastic optimal control theory, several constraints need to be considered. 
Especially in small devices \cite{fox_minimum-information_2016,fox_minimum-information_2016-1} and in biological systems \cite{li_iterative_2006,li_iterative_2007,nakamura_connection_2021,nakamura_optimal_2022,pezzotta_chemotaxis_2018,borra_optimal_2021}, 
incomplete information and memory limitation become predominant because their sensors are extremely noisy and their memory resources are severely limited. 
In order to account these constraints, memory-limited partially observable stochastic control (ML-POSC) has recently been proposed \cite{tottori_memory-limited_2022}. 
Because ML-POSC formulates the noisy observation and the limited memory explicitly, ML-POSC can directly take incomplete information and memory limitation into account in the stochastic optimal control problem. 

However, ML-POSC cannot be solved in the similar way as the conventional stochastic control, which is also called completely observable stochastic control (COSC). 
In COSC, the optimal control function depends only on the Hamilton-Jacobi-Bellman (HJB) equation, which is a time-backward partial differential equation given the terminal condition (Figure \ref{fig: COSC ML-POSC MFSC}(a)) \cite{yong_stochastic_1999}. 
Therefore, the optimal control function of COSC can be obtained by solving the HJB equation backward in time from the terminal condition, which is called the value iteration method \cite{kushner_numerical_1992,fleming_controlled_2006,puterman_markov_2014}.
In contrast, the optimal control function of ML-POSC depends not only on the HJB equation but also on the Fokker-Planck (FP) equation, which is a time-forward partial differential equation given the initial condition (Figure \ref{fig: COSC ML-POSC MFSC}(b)) \cite{tottori_memory-limited_2022}.
Because the HJB equation and the FP equation interact with each other through the optimal control function in ML-POSC, the optimal control function of ML-POSC cannot be obtained by the value iteration method. 

In order to propose an algorithm to ML-POSC, we firstly show that the system of HJB-FP equations can be interpreted via the Pontryagin's minimum principle on the probability density function space. 
The Pontryagin's minimum principle is one of the most representative approaches to the deterministic control, which converts the optimal control problem into the two-point boundary value problem of the forward state equation and the backward adjoint equation \cite{vinter_optimal_2010,lewis_optimal_2012,aschepkov_optimal_2016}. 
We show that the system of HJB-FP equations is an extension of the system of the state and adjoint equations from the deterministic control to the stochastic control.  

The system of HJB-FP equations also appears in the mean-field stochastic control (MFSC) \cite{bensoussan_mean_2013,carmona_probabilistic_2018,carmona_probabilistic_2018-1}. 
Although the relationship between the system of HJB-FP equations and the Pontryagin's minimum principle has been mentioned briefly in MFSC \cite{crisan_master_2014,bensoussan_master_2015,bensoussan_interpretation_2017}, its details have not yet been investigated. 
We resolve this problem by deriving the system of HJB-FP equations in the similar way as the Pontryagin's minimum principle. 

We then propose the forward-backward sweep method (FBSM) to ML-POSC. 
FBSM is an algorithm to compute the forward FP equation and the backward HJB equation alternately, which can be interpreted as an extension of the value iteration method. 
FBSM has been proposed in the Pontryagin's minimum principle of the deterministic control, which computes the forward state equation and the backward adjoint equation alternately \cite{krylov_method_1963,mitter_successive_1966,chernousko_method_1982}. 
Because FBSM is easy to implement, it has been used in many applications \cite{lenhart_optimal_2007,sharp_implementation_2021}.
However, the convergence of FBSM is not guaranteed in the deterministic control except for special cases \cite{hackbusch_numerical_1978,mcasey_convergence_2012} because the coupling of the state and adjoint equations is not limited to the optimal control function (Figure \ref{fig: COSC ML-POSC MFSC}(c)). 
In contrast, we show that the convergence of FBSM is generally guaranteed in ML-POSC because the coupling of HJB-FP equations is limited only to the optimal control function (Figure \ref{fig: COSC ML-POSC MFSC}(b)). 

FBSM corresponds to the fixed-point iteration method in MFSC \cite{carlini_semi-lagrangian_2013,carlini_fully_2014,carlini_semi-lagrangian_2015,lauriere_numerical_2021}. 
Although the fixed-point iteration method is the most basic algorithm in MFSC, its convergence is not guaranteed for the same reason as the deterministic control (Figure \ref{fig: COSC ML-POSC MFSC}(d)). 
Therefore, ML-POSC is a special class where FBSM or the fixed-point iteration method is guaranteed to converge. 

This paper is organized as follows: 
In Section \ref{sec: ML-POSC}, we formulate ML-POSC. 
In Section \ref{sec: OC}, we derive the system of HJB-FP equations from the viewpoint of the Pontryagin's minimum principle. 
In Section \ref{sec: FBSM}, we propose FBSM to ML-POSC and prove the convergence. 
In Section \ref{sec: LQG}, we apply FBSM to the linear-quadratic-Gaussian (LQG) problem. 
In Section \ref{sec: NE}, we verify the convergence of FBSM by numerical experiments.  
In Section \ref{sec: conclusion}, we discuss our work. 
In \ref{appendix: DC}, we briefly review the Pontryagin's minimum principle of the deterministic control. 
In \ref{appendix: MFC}, we show the Pontryagin's minimum principle of MFSC. 
In \ref{appendix: proof}, we show the proof in the main text. 

\begin{figure}[t]
\begin{center}
	\includegraphics[width=100mm]{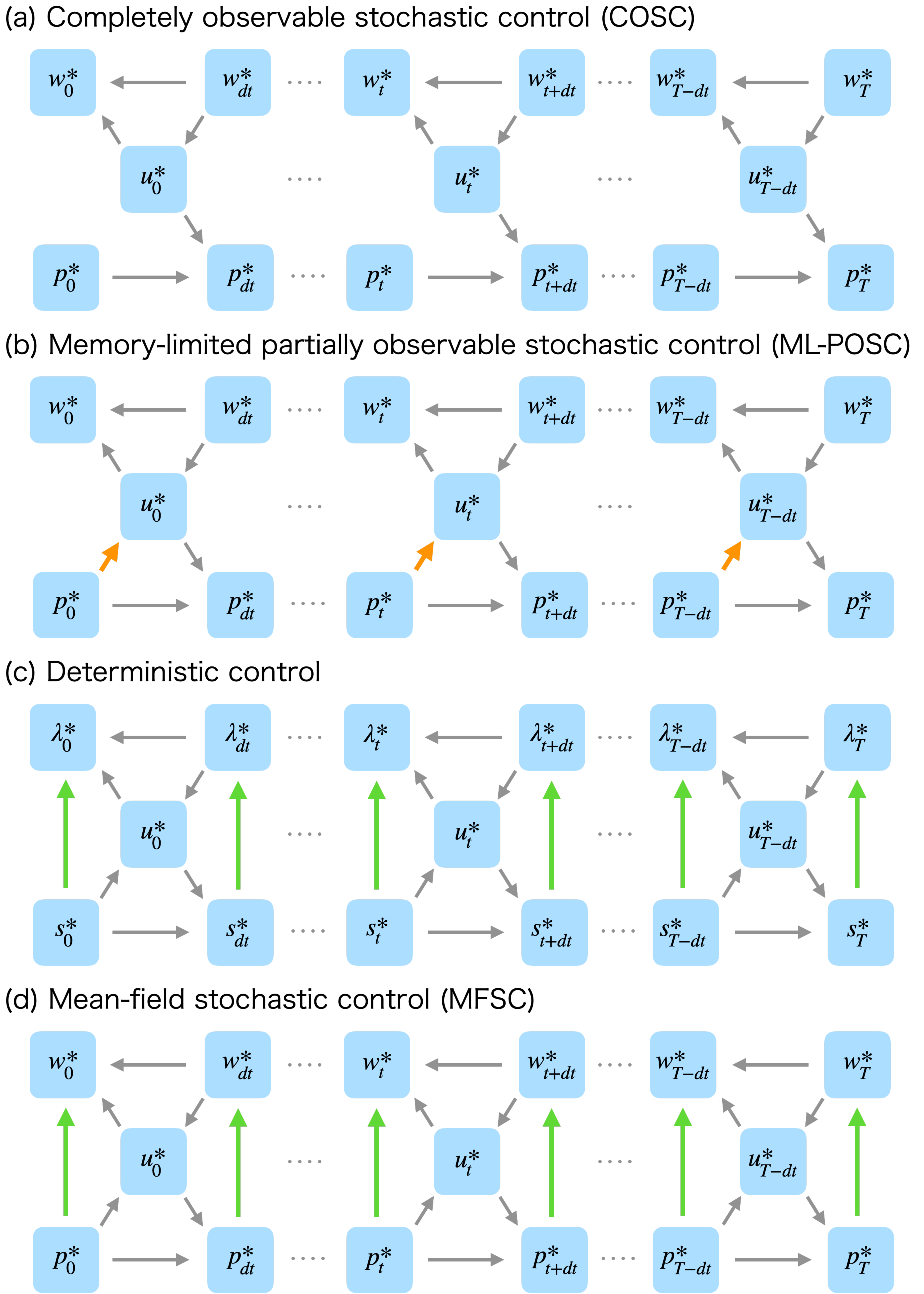}
	\caption{
	Schematic diagram of the relationship between the backward dynamics, the optimal control function, and the forward dynamics in 
	(a) COSC, (b) ML-POSC, (c) deterministic control, and (d) MFSC. 
	$w^{*}$, $p^{*}$, $\lambda^{*}$, and $s^{*}$ are the solutions of the HJB equation, the FP equation, the adjoint equation, and the state equation, respectively. 
	$u^{*}$ is the optimal control function. 
	(a) In COSC, because the optimal control function $u^{*}$ depends only on the HJB equation $w^{*}$, it can be obtained by solving the HJB equation $w^{*}$ backward in time from the terminal condition, which is called the value iteration method. 
	(b) In ML-POSC, because the optimal control function $u^{*}$ depends on the FP equation $p^{*}$ as well as the HJB equation $w^{*}$, it cannot be obtained by the value iteration method. 
	In this paper, we propose FBSM to ML-POSC, which computes the HJB equation $w^{*}$ and the FP equation $p^{*}$ alternately. 
	Because the coupling of the HJB equation $w^{*}$ and the FP equation $p^{*}$ is limited only to the optimal control function $u^{*}$, the convergence of FBSM is guaranteed in ML-POSC. 
	(c) In deterministic control, because the coupling of the adjoint equation $\lambda^{*}$ and the state equation $s^{*}$ is not limited to the optimal control function $u^{*}$, the convergence of FBSM is not guaranteed. 
	(d) In MFSC, because the coupling of the HJB equation $w^{*}$ and the FP equation $p^{*}$ is not limited to the optimal control function $u^{*}$, the convergence of FBSM is not guaranteed. 
	}
	\label{fig: COSC ML-POSC MFSC}
\end{center}
\end{figure}

\section{Memory-Limited Partially Observable Stochastic Control}\label{sec: ML-POSC}
In this section, we briefly review the formulation of ML-POSC \cite{tottori_memory-limited_2022}, which is the stochastic optimal control problem under incomplete information and memory limitation. 

\subsection{Problem Formulation}
In this subsection, we formulate ML-POSC \cite{tottori_memory-limited_2022}. 
The state of the system $x_{t}\in\mb{R}^{d_{x}}$ at time $t\in[0,T]$ evolves by the following stochastic differential equation (SDE): 
\begin{align}
	dx_{t}&=b(t,x_{t},u_{t})dt+\sigma(t,x_{t},u_{t})d\omega_{t},\label{eq: state SDE}
\end{align}
where $x_{0}$ obeys $p_{0}(x_{0})$, $u_{t}\in\mb{R}^{d_{u}}$ is the control, and $\omega_{t}\in\mb{R}^{d_{\omega}}$ is the standard Wiener process. 
In ML-POSC, because the controller cannot completely observe the state $x_{t}$, the observation $y_{t}\in\mb{R}^{d_{y}}$ is obtained instead of the state $x_{t}$, 
which evolves by the following SDE: 
\begin{align}
	dy_{t}&=h(t,x_{t})dt+\gamma(t)d\nu_{t},\label{eq: observation SDE}
\end{align}
where $y_{0}$ obeys $p_{0}(y_{0})$, and $\nu_{t}\in\mb{R}^{d_{\nu}}$ is the standard Wiener process. 
Furthermore, because the controller cannot completely memorize the observation history $y_{0:t}:=\{y_{\tau}|\tau\in[0,t]\}$, 
the observation history $y_{0:t}$ is compressed into the finite-dimensional memory $z_{t}\in\mb{R}^{d_{z}}$, 
which evolves by the following SDE: 
\begin{align}
	dz_{t}=c(t,z_{t},v_{t})dt+\kappa(t,z_{t},v_{t})dy_{t}+\eta(t,z_{t},v_{t})d\xi_{t}, 
	\label{eq: memory SDE}
\end{align}
where $z_{0}$ obeys $p_{0}(z_{0})$, $v_{t}\in\mb{R}^{d_{v}}$ is the control,  and $\xi_{t}\in\mb{R}^{d_{\xi}}$ is the standard Wiener process. 
The controller determines the state control $u_{t}$ and the memory control $v_{t}$ based on the memory $z_{t}$ as follows: 
\begin{align}
	u_{t}=u(t,z_{t}),\ v_{t}=v(t,z_{t}).\label{eq: control of ML-POSC}
\end{align}

The objective function of ML-POSC is given by the following expected cumulative cost function: 
\begin{align}	
	J[u,v]
	:=\mb{E}_{p(x_{0:T},y_{0:T},z_{0:T};u,v)}\left[\int_{0}^{T}f(t,x_{t},u_{t},v_{t})dt+g(x_{T})\right],
	\label{eq: OF of ML-POSC}
\end{align} 
where $f$ is the cost function, $g$ is the terminal cost function, $p(x_{0:T},y_{0:T},z_{0:T};u,v)$ is the probability of $x_{0:T}$, $y_{0:T}$, and $z_{0:T}$ given $u$ and $v$ as parameters, and $\mb{E}_{p}[\cdot]$ is the expectation with respect to the probability $p$. 

ML-POSC is the problem to find the optimal control functions $u^{*}$ and $v^{*}$ that minimize the expected cumulative cost function $J[u,v]$ as follows: 
\begin{align}	
	u^{*},v^{*}=\argmin_{u,v}J[u,v].
	\label{eq: ML-POSC}
\end{align} 

\subsection{Problem Reformulation}
Although the formulation of ML-POSC in the previous subsection is intuitive, it is inconvenient for further mathematical investigations. 
In order to resolve this problem, we reformulate ML-POSC in this subsection. 
The formulation in this subsection is simpler and more general than that in the previous subsection. 

We first define the extended state $s_{t}$ as follows: 
\begin{align}
	s_{t}:=\left(\begin{array}{c}
		x_{t}\\ z_{t}\\
	\end{array}\right)\in\mb{R}^{d_{s}},
\end{align}
where $d_{s}=d_{x}+d_{z}$. The extended state $s_{t}$ evolves by the following SDE: 
\begin{align}
	ds_{t}=\tilde{b}(t,s_{t},\tilde{u}_{t})dt+\tilde{\sigma}(t,s_{t},\tilde{u}_{t})d\tilde{\omega}_{t}, 
	\label{eq: extended state SDE}
\end{align}
where $s_{0}$ obeys $p_{0}(s_{0})$, $\tilde{u}_{t}\in\mb{R}^{d_{\tilde{u}}}$ is the control, and $\tilde{\omega}_{t}\in\mb{R}^{d_{\tilde{\omega}}}$ is the standard Wiener process. 
ML-POSC determines the control $\tilde{u}_{t}\in\mb{R}^{d_{\tilde{u}}}$ based on the memory $z_{t}$ as follows: 
\begin{align}
	\tilde{u}_{t}=\tilde{u}(t,z_{t}).\label{eq: control of GML-POSC}
\end{align}
The extended state SDE (\ref{eq: extended state SDE}) includes the previous SDEs (\ref{eq: state SDE}), (\ref{eq: observation SDE}), (\ref{eq: memory SDE}) as a special case because they can be represented as follows: 
\begin{align}
	&ds_{t}=\left(\begin{array}{c}
		b(t,x_{t},u_{t})\\
		c(t,z_{t},v_{t})+\kappa(t,z_{t},v_{t})h(t,x_{t})\\
	\end{array}\right)dt
	+\left(\begin{array}{ccc}
		\sigma(t,x_{t},u_{t})&O&O\\
		O&\kappa(t,z_{t},v_{t})\gamma(t)&\eta(t,z_{t},v_{t})\\
	\end{array}\right)
	\left(\begin{array}{c}
		d\omega_{t}\\
		d\nu_{t}\\
		d\xi_{t}\\
	\end{array}\right),
\end{align}
where $p_{0}(s_{0})=p_{0}(x_{0})p_{0}(z_{0})$. 

The objective function of ML-POSC is given by the following expected cumulative cost function: 
\begin{align}	
	J[\tilde{u}]:=\mb{E}_{p(s_{0:T};\tilde{u})}\left[\int_{0}^{T}\tilde{f}(t,s_{t},\tilde{u}_{t})dt+\tilde{g}(s_{T})\right].
	\label{eq: OF of GML-POSC}
\end{align}
where $\tilde{f}$ is the cost function, and $\tilde{g}$ is the terminal cost function.  
It is obvious that this objective function (\ref{eq: OF of GML-POSC}) is more general than that in the previous subsection (\ref{eq: OF of ML-POSC}). 

ML-POSC is the problem to find the optimal control function $\tilde{u}^{*}$ that minimizes the expected cumulative cost function $J[\tilde{u}]$ as follows: 
\begin{align}	
	\tilde{u}^{*}=\argmin_{\tilde{u}}J[\tilde{u}]. 
	\label{eq: GML-POSC}
\end{align}

In the following sections, we mainly consider the formulation of this subsection because it is simpler and more general than that in the previous subsection. 
Moreover, we omit $\tilde{\cdot}$ for the notational simplicity. 

\section{Pontryagin's Minimum Principle}\label{sec: OC}
If the control $u_{t}$ is determined based on the extended state $s_{t}$, i.e., $u_{t}=u(t,s_{t})$, ML-POSC is the same problem with COSC of the extended state, and its optimality conditions can be obtained in the conventional way \cite{yong_stochastic_1999}. 
In reality, however, because ML-POSC determines the control $u_{t}$ based only on the memory $z_{t}$, i.e., $u_{t}=u(t,z_{t})$, its optimality conditions cannot be obtained in the similar way as COSC. 
In the previous work \cite{tottori_memory-limited_2022}, the optimality conditions of ML-POSC were obtained by employing a mathematical technique of MFSC \cite{bensoussan_master_2015,bensoussan_interpretation_2017}. 

In this section, we obtain the optimality conditions of ML-POSC by employing the Pontryagin's minimum principle \cite{yong_stochastic_1999,vinter_optimal_2010,lewis_optimal_2012,aschepkov_optimal_2016} on the probability density function space (Figure \ref{fig: PMP}). 
The conventional approach in ML-POSC \cite{tottori_memory-limited_2022} and MFSC \cite{bensoussan_master_2015,bensoussan_interpretation_2017} can be interpreted as the conversion from the Bellman's dynamic programming principle to the Pontryagin's minimum principle on the probability density function space. 

In \ref{appendix: DC}, we briefly review the Pontryagin's minimum principle of the deterministic control. 
In this section, we obtain the optimality conditions of ML-POSC in the similar way as \ref{appendix: DC}. 
Furthermore, in \ref{appendix: MFC}, we obtain the optimality conditions of MFSC in the similar way as \ref{appendix: DC}. 
MFSC is more general than ML-POSC except for the partial observability. 
Especially, the expected Hamiltonian is non-linear with respect to the probability density function in MFSC while  it is linear in ML-POSC. 

\begin{figure}[t]
\begin{center}
	\includegraphics[width=135mm]{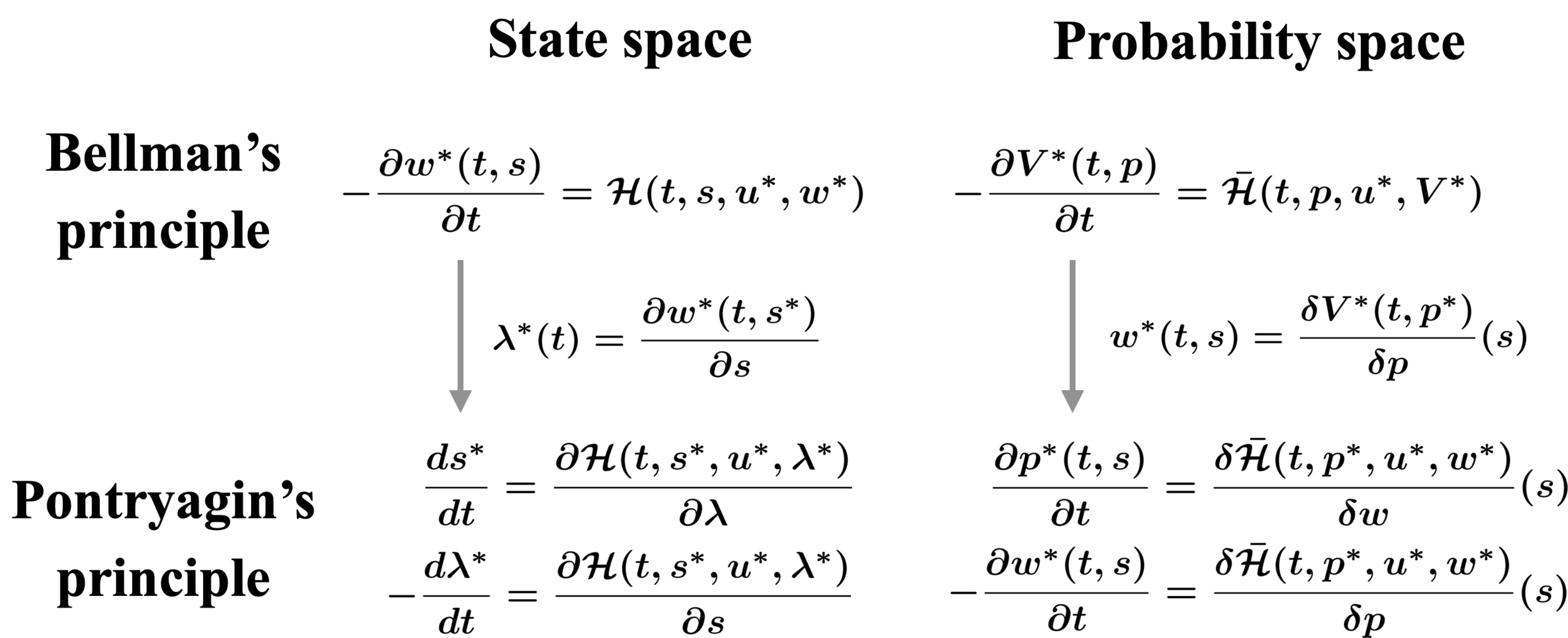}
	\caption{
	The relationship between the Bellman's dynamic programming principle (top) and the Pontryagin's minimum principle (bottom) on the state space (left) and on the probability density function space (right). 
	The left-hand side corresponds to the deterministic control, which is briefly reviewed in \ref{appendix: DC}. 
	The right-hand side corresponds to ML-POSC and MFSC, which are shown in Section \ref{sec: OC} and \ref{appendix: MFC}, respectively. 
	The conventional approach in ML-POSC \cite{tottori_memory-limited_2022} and MFSC \cite{bensoussan_master_2015,bensoussan_interpretation_2017} can be interpreted as the conversion from the Bellman's dynamic programming principle into the Pontryagin's minimum principle on the probability density function space. 
	}
	\label{fig: PMP}
\end{center}
\end{figure}

\subsection{Hamiltonian}
Before we show the optimality conditions, we define the Hamiltonian as follows: 
\begin{align}
	\mcal{H}\left(t,s,u,w\right):=f(t,s,u)+\mcal{L}_{u}w(t,s), 
	\label{eq: Hamiltonian}
\end{align}
where $^\forall w:[0,T]\times\mb{R}^{d_{s}}\to\mb{R}$, and $\mcal{L}_{u}$ is the backward diffusion operator, which is defined as follows: 
\begin{align}
	\mcal{L}_{u}w(t,s)&:=\sum_{i=1}^{d_{s}}b_{i}(t,s,u)\frac{\pl w(t,s)}{\pl s_{i}}+\frac{1}{2}\sum_{i,j=1}^{d_{s}}D_{ij}(t,s,u)\frac{\pl^{2} w(t,s)}{\pl s_{i}\pl s_{j}}, 
	\label{eq: backward diffusion operator}
\end{align}
where $D(t,s,u):=\sigma(t,s,u)\sigma^{\top}(t,s,u)$. 
The extended state SDE (\ref{eq: extended state SDE}) can be converted into the following Fokker-Planck (FP) equation: 
\begin{align}
	\frac{\pl p(t,s)}{\pl t}=\mcal{L}_{u}^{\dagger}p(t,s),
	\label{eq: FP eq}
\end{align}
where $p(0,s)=p_{0}(s)$, and $\mcal{L}_{u}^{\dag}$ is the forward diffusion operator, which is defined as follows: 
\begin{align}
	\mcal{L}_{u}^{\dag}p(t,s)&:=-\sum_{i=1}^{d_{s}}\frac{\pl (b_{i}(t,s,u)p(t,s))}{\pl s_{i}}+\frac{1}{2}\sum_{i,j=1}^{d_{s}}\frac{\pl^{2}(D_{ij}(t,s,u)p(t,s))}{\pl s_{i}\pl s_{j}}. 
	\label{eq: forward diffusion operator}
\end{align}
We note that $\mcal{L}_{u}^{\dag}$ is the conjugate of $\mcal{L}_{u}$ as follows: 
\begin{align}
	\int w(t,s)\mcal{L}_{u}^{\dag}p(t,s)ds=\int p(t,s)\mcal{L}_{u}w(t,s)ds,
	\label{eq: conjugate}
\end{align}
The following lemma shows the relationship between the objective function $J[u]$ and the Hamiltonian $\mcal{H}(t,s,u,w)$, 
which is significant for the optimality conditions.  

\begin{lemm}\label{lemm: J-J ML-POSC}
Let $^\forall u: [0,T]\times\mb{R}^{d_{z}}\to\mb{R}^{d_{u}}$ and $^\forall u': [0,T]\times\mb{R}^{d_{z}}\to\mb{R}^{d_{u}}$ be the arbitrary control functions, and let $p$ and $p'$ be the probability density functions of the extended state driven by $u$ and $u'$, respectively.  
Then $J[u]-J[u']$ satisfies the following equation: 
\begin{align}
	J[u]-J[u']
	&=\int_{0}^{T}\left(\mb{E}_{p(t,s)}\left[\mcal{H}(t,s,u,w')\right]-\mb{E}_{p(t,s)}\left[\mcal{H}(t,s,u',w')\right]\right)dt,
	\label{eq: J-J ML-POSC}
\end{align}
where $w'$ is the solution of the following Hamilton-Jacobi-Bellman (HJB) equation: 
\begin{align}
	-\frac{\pl w'(t,s)}{\pl t}=\mcal{H}\left(t,s,u',w'\right), 
	\label{eq: HJB eq}
\end{align}
where $w'(T,s)=g(s)$. 
\end{lemm}

\begin{proof}
The proof is shown in \ref{appendix: proof: lemm: J-J ML-POSC}. 
\end{proof}

\subsection{Necessary Condition}
We show the necessary condition of the optimal control function of ML-POSC that corresponds to the Pontryagin's minimum principle on the probability density function space. 

\begin{theo}\label{theo: NC ML-POSC}
In ML-POSC, the optimal control function $u^{*}$ satisfies the following equation:
\begin{align}
	u^{*}(t,z)=\argmin_{u}\mb{E}_{p_{t}^{*}(x|z)}\left[\mcal{H}\left(t,s,u,w^{*}\right)\right],\ a.s.\ ^\forall t\in[0,T],\ ^\forall z\in\mb{R}^{d_{z}},
	\label{eq: optimal control of ML-POSC}
\end{align}
where $p_{t}^{*}(x|z):=p^{*}(t,s)/\int p^{*}(t,s)dx$ is the conditional probability density function of the state $x$ given the memory $z$, and $p^{*}(t,s)$ is the solution of the FP equation: 
\begin{align}
	\frac{\pl p^{*}(t,s)}{\pl t}=\mcal{L}_{u^{*}}^{\dagger}p^{*}(t,s),
	\label{eq: optimal FP eq}
\end{align}
where $p^{*}(0,s)=p_{0}(s)$. 
$w^{*}(t,s)$ is the solution of the following HJB equation: 
\begin{align}
	-\frac{\pl w^{*}(t,s)}{\pl t}=\mcal{H}\left(t,s,u^{*},w^{*}\right), 
	\label{eq: optimal HJB eq}
\end{align}
where $w^{*}(T,s)=g(s)$. 
\end{theo}

\begin{proof}
The proof is shown in \ref{appendix: proof: theo: NC ML-POSC}. 
\end{proof}

In the Pontryagin's minimum principle of the deterministic control, the state equation (\ref{eq: optimal state equation}) and the adjoint equation (\ref{eq: optimal adjoint equation}) can be expressed by the derivatives of the Hamiltonian (\ref{appendix: DC}). 
Similarly, the system of HJB-FP equations (\ref{eq: optimal FP eq}) and (\ref{eq: optimal HJB eq}) can be expressed by the variations of the expected Hamiltonian
\begin{align}
	\bar{\mcal{H}}(t,p,u,w):=\mb{E}_{p(s)}\left[\mcal{H}\left(t,s,u,w\right)\right], 
	\label{eq: expected Hamiltonian}
\end{align}
as follows: 
\begin{align}
	\frac{\pl p^{*}(t,s)}{\pl t}&=\frac{\delta \bar{\mcal{H}}(t,p^{*},u^{*},w^{*})}{\delta w}(s),\label{eq: optimal FP eq PMP}\\
	-\frac{\pl w^{*}(t,s)}{\pl t}&=\frac{\delta \bar{\mcal{H}}(t,p^{*},u^{*},w^{*})}{\delta p}(s), \label{eq: optimal HJB eq PMP}
\end{align}
where $p^{*}(0,s)=p_{0}(s)$ and $w^{*}(T,s)=g(s)$. 
Therefore, the system of HJB-FP equations can be interpreted via the Pontryagin's minimum principle on the probability density function space. 

\subsection{Sufficient Condition}
The Pontryagin's minimum principle becomes a necessary and sufficient condition if the expected Hamiltonian is convex as follows: 

\begin{prop}\label{prop: SC ML-POSC}
Assume that the expected Hamiltonian $\bar{\mcal{H}}(t,p,u,w)$ is convex with respect to $p$ and $u$. 
If the control function $u^{*}$ satisfies (\ref{eq: optimal control of ML-POSC}), it is the optimal control function of ML-POSC. 
\end{prop}

\begin{proof}
The proof is shown in \ref{appendix: proof: theo: SC ML-POSC}. 
\end{proof}

\subsection{Relationship with Bellman's Dynamic Programming Principle}\label{sec: DPP ML-POSC}
From the Bellman's dynamic programming principle on the probability density function space \cite{tottori_memory-limited_2022}, the optimal control function of ML-POSC is given by the following equation: 
\begin{align}
	u^{*}(t,z,p)=\argmin_{u}\mb{E}_{p(x|z)}\left[\mcal{H}\left(t,s,u,\frac{\delta V^{*}(t,p)}{\delta p}(s)\right)\right].
	\label{eq: optimal control of ML-POSC DPP}
\end{align}
More specifically, the optimal control function of ML-POSC is given by $u^{*}(t,z)=u^{*}(t,z,p^{*})$, where $p^{*}$ is the solution of the FP equation (\ref{eq: optimal FP eq}). 
$V^{*}(t,p)$ is the value function on the probability density function space, which is the solution of the following Bellman equation: 
\begin{align}
	-\frac{\pl V^{*}(t,p)}{\pl t}=\mb{E}_{p(s)}\left[\mcal{H}\left(t,s,u^{*},\frac{\delta V^{*}(t,p)}{\delta p}(s)\right)\right],
	\label{eq: Bellman eq}
\end{align}
where $V^{*}(T,p)=\mb{E}_{p(s)}\left[g(s)\right]$. 
Because the Bellman equation (\ref{eq: Bellman eq}) is a functional differential equation, it cannot be solved even numerically. 
In order to resolve this problem, the previous work \cite{tottori_memory-limited_2022} converted the Bellman equation (\ref{eq: Bellman eq}) into the system of HJB-FP equations (\ref{eq: optimal FP eq}) and (\ref{eq: optimal HJB eq}) as follows: 

\begin{prop}[\cite{tottori_memory-limited_2022}]\label{prop: DPP ML-POSC}
We define $w^{*}(t,s)$ from $V^{*}(t,p)$ as follows: 
\begin{align}
	w^{*}(t,s):=\frac{\delta V^{*}(t,p^{*})}{\delta p}(s), 
\end{align}
where $p^{*}$ is the solution of FP equation (\ref{eq: optimal FP eq}). 
$w^{*}$ satisfies HJB equation (\ref{eq: optimal HJB eq}) from the Bellman equation (\ref{eq: Bellman eq}). 
\end{prop}

\begin{proof}
The proof is shown in \cite{tottori_memory-limited_2022}. 
\end{proof}

This approach \cite{tottori_memory-limited_2022} can be interpreted as the conversion from the Bellman's dynamic programming principle to the Pontryagin's minimum principle on the probability density function space. 

\subsection{Relationship with Completely Observable Stochastic Control}\label{sec: COSC}
In COSC of the extended state, the control $u_{t}$ is determined based on the extended state  $s_{t}$, i.e., $u_{t}=u(t,s_{t})$. 
In COSC of the extended state, the Pontryagin's minimum principle on the probability density function space is given by the following proposition, which is a necessary condition of the optimal control function: 

\begin{prop}\label{prop: NC COSC}
In COSC, the optimal control function $u^{*}$ satisfies the following equation: 
\begin{align}
	u^{*}(t,s)=\argmin_{u}\mcal{H}\left(t,s,u,w^{*}\right),\ a.s.\ ^\forall t\in[0,T],\ ^\forall s\in\mb{R}^{d_{s}},
	\label{eq: optimal control of COSC}
\end{align}
where $w^{*}(t,s)$ is the solution of HJB equation (\ref{eq: optimal HJB eq}). 
\end{prop}

\begin{proof}
The proof is almost the same with Theorem \ref{theo: NC ML-POSC}. 
\end{proof}

While the optimal control function of ML-POSC depends on the FP equation and the HJB equation, the optimal control function of COSC depends only on the HJB equation. 
From this nice property of COSC, Proposition \ref{prop: NC COSC} is a necessary and sufficient condition without the convexity of the expected Hamiltonian as follows: 

\begin{prop}\label{prop: SC COSC}
If the control function $u^{*}$ satisfies (\ref{eq: optimal control of COSC}), it is the optimal control function of COSC. 
\end{prop}

\begin{proof}
The proof is shown in \ref{appendix: proof: theo: SC COSC}. 
\end{proof}

Proposition \ref{prop: SC COSC} is consistent with the conventional result of COSC \cite{yong_stochastic_1999}. 
Unlike ML-POSC and MFSC, COSC can be solved by the Bellman's dynamic programming principle on the state space. 
In COSC, the Pontryagin's minimum principle on the probability density function space is equivalent with the Bellman's dynamic programming principle on the state space. 
Because the Bellman's dynamic programming principle on the state space is a necessary and sufficient condition, the Pontryagin's minimum principle on the probability density function space also becomes a necessary and sufficient condition.

\section{Forward-Backward Sweep Method}\label{sec: FBSM}
In this section, we propose FBSM to ML-POSC, and then prove the convergence. 
While the convergence of FBSM is not guaranteed in the deterministic control \cite{krylov_method_1963,chernousko_method_1982,lenhart_optimal_2007,mcasey_convergence_2012} and MFSC \cite{carlini_semi-lagrangian_2013,carlini_fully_2014,carlini_semi-lagrangian_2015,lauriere_numerical_2021}, it is guaranteed in ML-POSC because the coupling of HJB-FP equations is limited only to the optimal control function in ML-POSC. 

\subsection{Forward-Backward Sweep Method}
In this subsection, we propose FBSM to ML-POSC, which is summarized in Algorithm \ref{alg: FBSM}. 
FBSM is an algorithm to compute the forward FP equation and the backward HJB equation alternately. 
More specifically, in the initial step of FBSM, we initialize the control function $u_{0:T-dt}^{0}$ and compute the FP equation $p_{0:T}^{0}$ forward in time from the initial condition. 
In the backward step, we compute the HJB equation $w_{0:T}^{1}$ backward in time from the terminal condition and simultaneously compute the control function $u_{0:T-dt}^{1}$ by minimizing the conditional expected Hamiltonian. 
In the forward step, we compute the FP equation $p_{0:T}^{2}$ forward in time from the initial condition and simultaneously compute the control function $u_{0:T-dt}^{2}$ in the similar way as the backward step. 
By iterating the backward step and the forward step, the objective function $J[u_{0:T-dt}^{k}]$ monotonically decreases and finally converges to the local minimum at which the control function $u_{0:T-dt}^{k}$ satisfies the Pontryagin's minimum principle (Theorem \ref{theo: NC ML-POSC}). 

The Pontryagin's minimum principle is a necessary condition of the optimal control function, not a sufficient condition. 
Therefore, the control function obtained by FBSM is not necessarily the global optimum except in the case where the expected Hamiltonian is convex (Proposition \ref{prop: SC ML-POSC}). 
However, the control function obtained by FBSM is expected to be superior to most of control functions because it is locally optimal. 

\begin{algorithm}[t]
\caption{Forward-Backward Sweep Method (FBSM)}
\label{alg: FBSM}
\begin{algorithmic}
\STATE //--- Initial step ---//
\STATE $k\leftarrow0$
\STATE $p_{0}^{k}(s)\leftarrow p_{0}(s)$
\FOR{$t=0$ to $T-dt$}
	\STATE Initialize $u_{t}^{k}(z)$
	\STATE $p_{t+dt}^{k}(s)\leftarrow p_{t}^{k}(s)+\mcal{L}_{u_{t}^{k}}^{\dag}p_{t}^{k}(s)dt$
\ENDFOR
\WHILE{$J[u_{0:T-dt}^{k}]$ do not converge}
	\IF{$k$ is even}
		\STATE //--- Backward step ---//
		\STATE $w_{T}^{k+1}(s)\leftarrow g(s)$
		\FOR{$t=T-dt$ to $0$}
			\STATE $u_{t}^{k+1}(z)\leftarrow \argmin_{u}\mb{E}_{p_{t}^{k}(x|z)}\left[\mcal{H}(t,s,u,w_{t+dt}^{k+1})\right]$
			\STATE $w_{t}^{k+1}(s)\leftarrow w_{t+dt}^{k+1}(s)+\mcal{H}(t,s,u_{t}^{k+1},w_{t+dt}^{k+1})dt$
		\ENDFOR
	\ELSE
		\STATE //--- Forward step ---//
		\STATE $p_{0}^{k+1}(s)\leftarrow p_{0}(s)$
		\FOR{$t=0$ to $T-dt$}
			\STATE $u_{t}^{k+1}(z)\leftarrow \argmin_{u}\mb{E}_{p_{t}^{k+1}(x|z)}\left[\mcal{H}(t,s,u,w_{t+dt}^{k})\right]$
			\STATE $p_{t+dt}^{k+1}(s)\leftarrow p_{t}^{k+1}(s)+\mcal{L}_{u_{t}^{k+1}}^{\dag}p_{t}^{k+1}(s)dt$
		\ENDFOR
	\ENDIF
	\STATE $k\leftarrow k+1$
\ENDWHILE
\RETURN $u_{0:T-dt}^{k}$
\end{algorithmic}
\end{algorithm}

\subsection{Convergence}
FBSM has been used in the deterministic control \cite{krylov_method_1963,chernousko_method_1982,lenhart_optimal_2007,mcasey_convergence_2012} and MFSC \cite{carlini_semi-lagrangian_2013,carlini_fully_2014,carlini_semi-lagrangian_2015,lauriere_numerical_2021}. 
However, the convergence of FBSM for these cases is not guaranteed because the backward dynamics depends on the forward dynamics even without the optimal control function (Figure \ref{fig: COSC ML-POSC MFSC}(c,d)). 
In contrast, the convergence of FBSM is guaranteed in ML-POSC because the backward HJB equation does not depend on the forward FP equation without the optimal control function (Figure \ref{fig: COSC ML-POSC MFSC}(b)). 

We firstly show the following lemma, which plays a key role in the convergence of FBSM: 
\begin{lemm}\label{lemm: current optimal control}
Given $u_{0:t-dt,t+dt:T-dt}:=\{u_{0},...,u_{t-dt},u_{t+dt},...,u_{T-dt}\}$, $u_{t}^{*}$ is defined by
\begin{align}
	u_{t}^{*}(z):=\argmin_{u_{t}}\mb{E}_{p_{t}(x|z)}\left[\mcal{H}\left(t,s,u_{t},w_{t+dt}\right)\right],
	\label{eq: current optimal control of ML-POSC H}
\end{align}
where $p_{t}(x|z):=p_{t}(s)/\int p_{t}(s)dx$ is the conditional probability density function of the state $x$ given the memory $z$, and $p_{t}(s)$ is the solution of the following FP equation: 
\begin{align}
	p_{t+dt}(s)=p_{t}(s)+\mcal{L}_{u_{t}}^{\dagger}p_{t}(s)dt,
	\label{eq: time-discretized FP eq}
\end{align}
where $p_{0}(s)$. 
$w_{t}(s)$ is the solution of the following HJB equation: 
\begin{align}
	w_{t}(s)=w_{t+dt}(s)+\mcal{H}\left(t,s,u_{t},w_{t+dt}\right)dt, 
	\label{eq: time-discretized HJB eq}
\end{align}
where $w_{T}(s)=g(s)$. 
Then $u_{t}^{*}$ satisfies the following equation in ML-POSC: 
\begin{align}
	u_{t}^{*}=\argmin_{u_{t}}J[u_{0:T-dt}]. 
	\label{eq: current optimal control of ML-POSC J}
\end{align}
\end{lemm}

\begin{proof}
This lemma is proven by the time discretization method in \ref{appendix: proof: lemm: current optimal control ver 1} and also by the similar way as the Pontyragin's minimum principle in \ref{appendix: proof: lemm: current optimal control ver 2}. 
While the former is more intuitive, the latter is more useful for comparing ML-POSC with the deterministic control and MFSC. 
\end{proof}

Importantly, $w_{t+dt}$ does not depend on $u_{t}$ in ML-POSC (Figure \ref{fig: POSC MFC pert}(a)) while $\lambda_{t+dt}$ and $w_{t+dt}$ depend on $u_{t}$ in the deterministic control (Figure \ref{fig: POSC MFC pert}(b)) and MFSC (Figure \ref{fig: POSC MFC pert}(c)), respectively. 
Therefore, $u_{t}^{*}$ can be obtained without modifying $w_{t+dt}$ in ML-POSC, which is essentially different from the deterministic control and MFSC. 
From this nice property of ML-POSC, FBSM becomes monotonic as follows: 

\begin{lemm}\label{lemm: tight monotonicity of FBSM}
In FBSM of ML-POSC, the objective function is monotonically non-increasing with respect to the update of the control function at each time step. 
More specifically, 
\begin{align}
	J[u_{0:t-dt}^{k},u_{t:T-dt}^{k+1}]\leq J[u_{0:t}^{k},u_{t+dt:T-dt}^{k+1}]
	\label{eq: tight monotonicity of backward step}
\end{align}
is satisfied in the backward step, and 
\begin{align}
	J[u_{0:t-dt}^{k+1},u_{t:T-dt}^{k}]\geq J[u_{0:t}^{k+1},u_{t+dt:T-dt}^{k}]
	\label{eq: tight monotonicity of forward step}
\end{align}
is satisfied in the forward step.
\end{lemm}

\begin{proof}
The proof is shown in \ref{appendix: proof: lemm: tight monotonicity of FBSM}. 
\end{proof}

\begin{theo}\label{theo: monotonicity of FBSM}
In FBSM of ML-POSC, the objective function is monotonically non-increasing with respect to the update of the control function at each iteration step. 
More specifically, 
\begin{align}
	J[u_{0:T-dt}^{k+1}]\leq J[u_{0:T-dt}^{k}]
\end{align}
is satisfied. 
\end{theo}

\begin{proof}
It is obvious from Lemma \ref{lemm: tight monotonicity of FBSM}. 
\end{proof}

Therefore, if $J[u_{0:T-dt}]$ has a lower bound, FBSM is guaranteed to converge to the local minimum in ML-POSC. 
At the local minimum, the Pontryagin's minimum principle is satisfied as follows: 

\begin{theo}\label{theo: convergence of FBSM}
Assume that if the candidate of $u_{t}^{k+1}(z)$ includes $u_{t}^{k}(z)$, then set $u_{t}^{k+1}(z)$ at $u_{t}^{k}(z)$. If
\begin{align}
	J[u_{0:T-dt}^{k+1}]=J[u_{0:T-dt}^{k}]
\end{align}
holds, $u_{0:T-dt}^{k+1}$ satisfies the Pontryagin's minimum principle (Theorem \ref{theo: NC ML-POSC}), which is a necessary condition of the optimal control function. 
\end{theo}

\begin{proof}
The proof is shown in \ref{appendix: proof: theo: convergence of FBSM}. 
\end{proof}

Therefore, unlike the deterministic control and MFSC, in FBSM of ML-POSC, the objective function $J[u_{0:T-dt}^{k}]$ monotonically decreases and finally converges to the local minimum at which the control function $u_{0:T-dt}^{k}$ satisfies the Pontryagin's minimum principle. 

\begin{figure}[t]
\begin{center}
	\includegraphics[width=100mm]{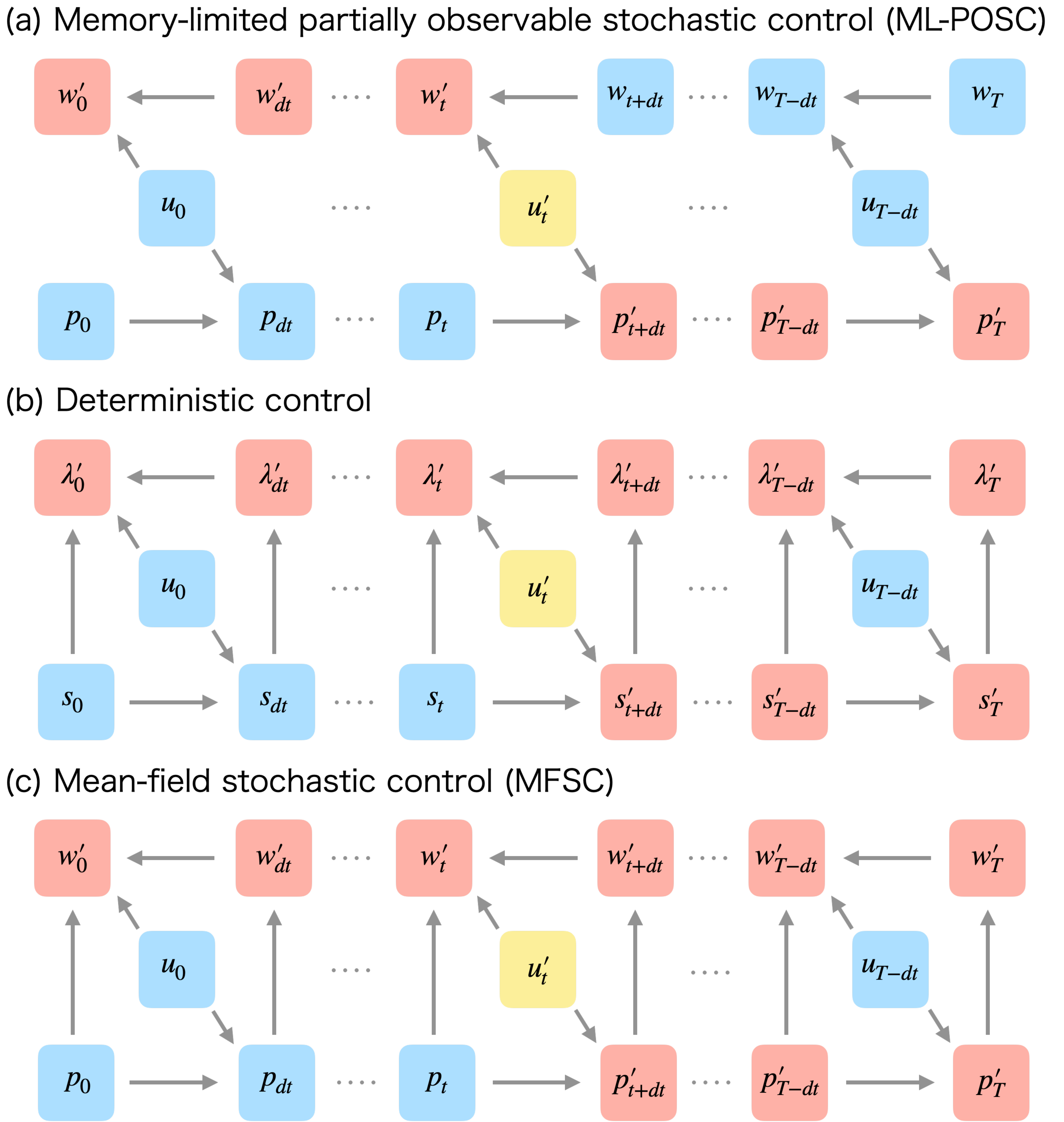}
	\caption{
	Schematic diagram of the effect of updating the control function on the forward and backward dynamics in 
	(a) ML-POSC, 
	(b) deterministic control, 
	and (c) MFSC. 
	$w_{0:T}$, $p_{0:T}$, $\lambda_{0:T}$, and $s_{0:T}$ are the solutions of the HJB equation, the FP equation, the adjoint equation, and the state equation, respectively. 
	$u_{0:T-dt}$ is the given control function. 
	(a) In ML-POSC, the update from $u_{t}$ to $u_{t}'$ does not change $p_{0:t}$ and $w_{t+dt:T}$. 
	From this property, the convergence of FBSM is guaranteed.  
	(b) In deterministic control, the update from $u_{t}$ to $u_{t}'$ changes $\lambda_{t+dt:T}$ to $\lambda_{t+dt:T}'$ because the adjoint equation depends on the state equation. 
	Because FBSM does not take the change of $\lambda_{t+dt:T}$ into account, the convergence of FBSM is not guaranteed. 
	(c) In MFSC, the update from $u_{t}$ to $u_{t}'$ changes $w_{t+dt:T}$ to $w_{t+dt:T}'$ because the HJB equation depends on the FP equation. 
	Because FBSM does not take the change of $w_{t+dt:T}$ into account, the convergence of FBSM is not guaranteed. 
	}
	\label{fig: POSC MFC pert}
\end{center}
\end{figure}

\section{Linear-Quadratic-Gaussian Problem}\label{sec: LQG}
In this section, we apply FBSM into the LQG problem of ML-POSC \cite{tottori_memory-limited_2022}. 
In the LQG problem of ML-POSC, the system of HJB-FP equations is reduced from partial differential equations to ordinary differential equations. 

\subsection{Problem Formulation}
In the LQG problem of ML-POSC, the extended state SDE (\ref{eq: extended state SDE}) is given as follows \cite{tottori_memory-limited_2022}:
\begin{align}
	ds_{t}=\left(A(t)s_{t}+B(t)u_{t}\right)dt+\sigma(t)d\omega_{t},\label{SDE of LQG}
\end{align}
where $s_{0}$ obeys the Gaussian distribution $p_{0}(s_{0}):=\mcal{N}\left(s_{0}\left|\mu_{0},\Lambda_{0}\right.\right)$ where $\mu_{0}$ is the mean vector and  $\Lambda_{0}$ is the precision matrix. 
The objective function (\ref{eq: OF of GML-POSC}) is given as follows:  
\begin{align}
	J[u]:=\mb{E}_{p(s_{0:T};u)}\left[\int_{0}^{T}\left(s_{t}^{\top}Q(t)s_{t}+u_{t}^{\top}R(t)u_{t}\right)dt+s_{T}^{\top}Ps_{T}\right],\label{OF of LQG}
\end{align}
where $Q(t)\succeq O$, $R(t)\succ O$, and $P\succeq O$. 
The LQG problem of ML-POSC is the problem to find the optimal control function $u^{*}$ that minimize the objective function $J[u]$ as follows: 
\begin{align}	
	u^{*}=\argmin_{u}J[u].
\end{align} 

\subsection{Pontryagin's Minimum Principle}
In the LQG problem of ML-POSC, the Pontryagin's minimum principle (Theorem \ref{theo: NC ML-POSC}) can be calculated as follows \cite{tottori_memory-limited_2022}:
 
\begin{prop}[\cite{tottori_memory-limited_2022}]\label{prop: optimal control of LQG in ML-POSC}
In the LQG problem of ML-POSC, the optimal control function satisfies the following equation: 
\begin{align}
	u^{*}(t,z)=-R^{-1}B^{\top}\left(\Pi K(\Lambda)(s-\mu)+\Psi\mu\right),
	\label{eq: optimal control of LQG}
\end{align}
where $K(\Lambda)$ is defined as follows:
\begin{align}
	K(\Lambda):=\left(\begin{array}{cc}
		O&\Lambda_{xx}^{-1}\Lambda_{xz}\\ 
		O&I\\
	\end{array}\right),
	\label{eq: inference gain}
\end{align}
where $\mu(t)$ and $\Lambda(t)$ are the mean vector and the precision matrix of the extended state, respectively, which correspond to the solution of the FP equation (\ref{eq: optimal FP eq}).
We note that $\mb{E}_{p_{t}(z|x)}\left[s\right]=K(\Lambda)(s-\mu)+\mu$ is satisfied. 
$\mu(t)$ and $\Lambda(t)$ are the solutions of the following ordinary differential equations (ODEs): 
\begin{align}
	\dot{\mu}&=\left(A-BR^{-1}B^{\top}\Psi\right)\mu,\label{eq: ODE of mu}\\
	\dot{\Lambda}&=-\left(A-BR^{-1}B^{\top}\Pi K(\Lambda)\right)^{\top}\Lambda-\Lambda\left(A-BR^{-1}B^{\top}\Pi K(\Lambda)\right)-\Lambda\sigma\sigma^{\top}\Lambda,\label{eq: ODE of Lambda}
\end{align}
where $\mu(0)=\mu_{0}$ and $\Lambda(0)=\Lambda_{0}$. 
$\Psi(t)$ and $\Pi(t)$ are the control gain matrices of the deterministic and stochastic extended state, respectively, which correspond to the solution of the HJB equation (\ref{eq: optimal HJB eq}).
$\Psi(t)$ and $\Pi(t)$ are the solutions of the following ODEs: 
\begin{align}
	-\dot{\Psi}&=Q+A^{\top}\Psi+\Psi A -\Psi BR^{-1}B^{\top}\Psi,\label{eq: ODE of Psi}\\
	-\dot{\Pi}&=Q+A^{\top}\Pi+\Pi A-\Pi BR^{-1}B^{\top}\Pi+(I-K(\Lambda))^{\top}\Pi BR^{-1}B^{\top}\Pi (I-K(\Lambda)), \label{eq: ODE of Pi}
\end{align}
where $\Psi(T)=\Pi(T)=P$. 
The ODE of $\Psi$ (\ref{eq: ODE of Psi}) is the Riccati equation \cite{yong_stochastic_1999,bensoussan_estimation_2018}, which also appears in the LQG problem of COSC. 
In contrast, the ODE of $\Pi$ (\ref{eq: ODE of Pi}) is the partially observable Riccati equation \cite{tottori_memory-limited_2022}, which appears only in the LQG problem of ML-POSC. 
\end{prop}

\begin{proof}
The proof is shown in \cite{tottori_memory-limited_2022}.
\end{proof}

The ODE of $\Psi$ (\ref{eq: ODE of Psi}) can be solved backward in time from the terminal condition. 
Using $\Psi$, the ODE of $\mu$ (\ref{eq: ODE of mu}) can be solved forward in time from the initial condition. 
In contrast, the ODEs of $\Pi$ (\ref{eq: ODE of Pi}) and $\Lambda$ (\ref{eq: ODE of Lambda}) cannot be solved in the similar way as the ODEs of $\Psi$ (\ref{eq: ODE of Psi}) and $\mu$ (\ref{eq: ODE of mu}) because they interact with each other, which is the similar problem with the system of HJB-FP equations. 

\subsection{Forward-Backward Sweep Method}
In the FBSM of the LQG problem, the FP equation (\ref{eq: optimal FP eq}) and the HJB equation (\ref{eq: optimal HJB eq}) are reduced to the ODEs of $\Lambda$ (\ref{eq: ODE of Lambda}) and $\Pi$ (\ref{eq: ODE of Pi}), respectively, which is summarized in Algorithm \ref{alg: FBSM LQG}. 
$\mcal{F}(\Lambda,\Pi)$ and $\mcal{G}(\Lambda,\Pi)$ are defined by the right-hand sides of (\ref{eq: ODE of Lambda}) and (\ref{eq: ODE of Pi}), respectively, as follows: 
\begin{align}
	\mcal{F}(\Lambda,\Pi)&:=-\left(A-BR^{-1}B^{\top}\Pi K(\Lambda)\right)^{\top}\Lambda-\Lambda\left(A-BR^{-1}B^{\top}\Pi K(\Lambda)\right)-\Lambda\sigma\sigma^{\top}\Lambda,\nonumber\\
	\mcal{G}(\Lambda,\Pi)&:=Q+A^{\top}\Pi+\Pi A-\Pi BR^{-1}B^{\top}\Pi+(I-K(\Lambda))^{\top}\Pi BR^{-1}B^{\top}\Pi (I-K(\Lambda)).\nonumber
\end{align}
The following proposition shows that FBSM is reduced from Algorithm \ref{alg: FBSM} to Algorithm \ref{alg: FBSM LQG} in the LQG problem: 

\begin{prop}\label{prop: FBSM for LQG problem}
Algorithm \ref{alg: FBSM LQG} is the FBSM of the LQG problem when the control function is initialized by
\begin{align}
	u^{0}(t,z)=-R^{-1}B^{\top}\left(\Pi^{0} K(\Lambda^{0})(s-\mu)+\Psi\mu\right),
	\label{eq: FBSM LQG the initial control function}
\end{align}
where $\Pi^{0}$ is arbitrary and $\Lambda^{0}$ is the solution of $\dot{\Lambda}^{0}=\mcal{F}(\Lambda^{0},\Pi^{0})$ given $\Lambda^{0}(0)=\Lambda_{0}$. 
\end{prop}

\begin{proof}
The proof is shown in \ref{appendix: proof: theo: FBSM for LQG problem}. 
\end{proof}

\begin{algorithm}[t]
\caption{Forward-Backward Sweep Method (FBSM) in the LQG problem}
\label{alg: FBSM LQG}
\begin{algorithmic}
\STATE //--- Initial step ---//
\STATE $k\leftarrow0$
\STATE $\Lambda_{0}^{k}\leftarrow \Lambda_{0}$
\FOR{$t=0$ to $T-dt$}
	\STATE Initialize $\Pi_{t+dt}^{k}$
	\STATE $\Lambda_{t+dt}^{k}\leftarrow\Lambda_{t}^{k}+\mcal{F}(\Lambda_{t}^{k},\Pi_{t+dt}^{k})dt$
\ENDFOR
\WHILE{$J[u_{0:T-dt}^{k}]$ do not converge}
	\IF{$k$ is even}
		\STATE //--- Backward step ---//
		\STATE $\Pi_{T}^{k+1}\leftarrow P$
		\FOR{$t=T-dt$ to $0$}
			\STATE $\Pi_{t}^{k+1}\leftarrow\Pi_{t+dt}^{k+1}+\mcal{G}(\Lambda_{t}^{k},\Pi_{t+dt}^{k+1})dt$
		\ENDFOR
	\ELSE
		\STATE //--- Forward step ---//
		\STATE $\Lambda_{0}^{k+1}\leftarrow \Lambda_{0}$
		\FOR{$t=0$ to $T-dt$}
			\STATE $\Lambda_{t+dt}^{k+1}\leftarrow\Lambda_{t}^{k+1}+\mcal{F}(\Lambda_{t}^{k+1},\Pi_{t+dt}^{k})dt$
		\ENDFOR
	\ENDIF
	\STATE $k\leftarrow k+1$
\ENDWHILE
\RETURN $u_{0:T-dt}^{k}$
\end{algorithmic}
\end{algorithm}

\section{Numerical Experiments}\label{sec: NE}
In this section, we verify the convergence of FBSM in ML-POSC by conducting the numerical experiments of the LQG and non-LQG problems. 
The setting of the numerical experiments is the same with the previous work \cite{tottori_memory-limited_2022}.  

\subsection{LQG Problem}
In this subsection, we verify the convergence of FBSM in ML-POSC by conducting the numerical experiment of the LQG problem. 
We consider the state $x_{t}\in\mb{R}$, the observation $y_{t}\in\mb{R}$, and the memory $z_{t}\in\mb{R}$, which evolve by the following SDEs: 
\begin{align}
	dx_{t}&=\left(x_{t}+u_{t}\right)dt+d\omega_{t},\label{eq: state SDE LQG NE}\\
	dy_{t}&=x_{t}dt+d\nu_{t},\label{eq: observation SDE LQG NE}\\
	dz_{t}&=v_{t}dt+dy_{t},\label{eq: memory SDE LQG NE}
\end{align}
where $x_{0}$ and $z_{0}$ obey the standard Gaussian distributions, $y_{0}$ is an arbitrary real number, $\omega_{t}\in\mb{R}$ and $\nu_{t}\in\mb{R}$ are independent standard Wiener processes, $u_{t}=u(t,z_{t})\in\mb{R}$ and $v_{t}=v(t,z_{t})\in\mb{R}$ are the controls. 
The objective function to be minimized is given as follows: 
\begin{align}	
	J[u,v]:=\mb{E}_{p(x_{0:10},y_{0:10},z_{0:10};u,v)}\left[\int_{0}^{10}\left(x_{t}^{2}+u_{t}^{2}+v_{t}^{2}\right)dt\right].
	\label{eq: OF of ML-POSC LQG NE}
\end{align} 
Therefore, the objective of this problem is to minimize the state variance by the small state and memory controls. 

This problem corresponds to the LQG problem, which is defined by (\ref{SDE of LQG}) and (\ref{OF of LQG}). 
By defining $s_{t}:=(x_{t},z_{t})\in\mb{R}^{2}$, $\tilde{u}_{t}:=(u_{t},v_{t})\in\mb{R}^{2}$, and $\tilde{\omega}_{t}:=(\omega_{t},\nu_{t})\in\mb{R}^{2}$, the SDEs (\ref{eq: state SDE LQG NE}), (\ref{eq: observation SDE LQG NE}), (\ref{eq: memory SDE LQG NE}) can be rewritten as follows: 
\begin{align}
	&ds_{t}=\left(\left(\begin{array}{cc}
		1&0\\
		1&0\\
	\end{array}\right)s_{t}
	+\tilde{u}_{t}\right)dt
	+d\tilde{\omega}_{t}, 
\end{align}
which corresponds to (\ref{SDE of LQG}). 
Furthermore, the objective function (\ref{eq: OF of ML-POSC LQG NE}) can be rewritten as follows: 
\begin{align}	
	J[\tilde{u}]:=\mb{E}_{p(s_{0:10};\tilde{u})}\left[\int_{0}^{10}\left(
	s_{t}^{\top}\left(\begin{array}{cc}
		1&0\\
		0&0\\
	\end{array}\right)s_{t}
	+\tilde{u}_{t}^{\top}\tilde{u}_{t}\right)dt\right],
\end{align} 
which corresponds to (\ref{OF of LQG}). 

We apply the FBSM of the LQG problem (Algorithm \ref{alg: FBSM LQG}) to this problem. 
$\Pi^{0}(t)$ is initialized by $\Pi^{0}(t)=O$. 
In order to solve the ODEs of $\Pi^{k}(t)$ and $\Lambda^{k}(t)$, we use the Runge-Kutta method. 
Figure \ref{fig: LQG Pi and Lambda} shows the control gain matrix $\Pi^{k}(t)\in\mb{R}^{2\times2}$ and the precision matrix $\Lambda^{k}(t)\in\mb{R}^{2\times2}$ obtained by FBSM. 
The color of the curve represents the iteration $k$. 
The darkest curve corresponds to the first iteration $k=0$, and the brightest curve corresponds to the last iteration $k=50$. 
Importantly, $\Pi^{k}(t)$ and $\Lambda^{k}(t)$ converge with respect to the iteration $k$. 

Figure \ref{fig: LQG performance}(a) shows the objective function $J[u^{k}]$ with respect to the iteration $k$. 
The objective function $J[u^{k}]$ monotonically decreases with respect to the iteration $k$, which is consistent with Theorem \ref{theo: monotonicity of FBSM}. 
This monotonicity of FBSM is the nice property of ML-POSC that is not guaranteed in the deterministic control and MFSC. 
The objective function $J[u^{k}]$ finally converges, and its $u^{k}$ satisfies the Pontryagin's minimum principle from Theorem \ref{theo: convergence of FBSM}. 

Figure \ref{fig: LQG performance}(b,c,d) compares the performance of the control function $u^{k}$ at the first iteration $k=0$ and the last iteration $k=50$ by conducting the stochastic simulation. 
At the first iteration $k=0$, the distributions of the state and the memory are unstable, and the cumulative cost diverges. 
In contrast, at the last iteration $k=50$, the distributions of the state and the memory are stable, and the cumulative cost is smaller. 
This result indicates that FBSM improves the performance in ML-POSC. 

\begin{figure}[t]
\begin{center}
	\includegraphics[width=135mm]{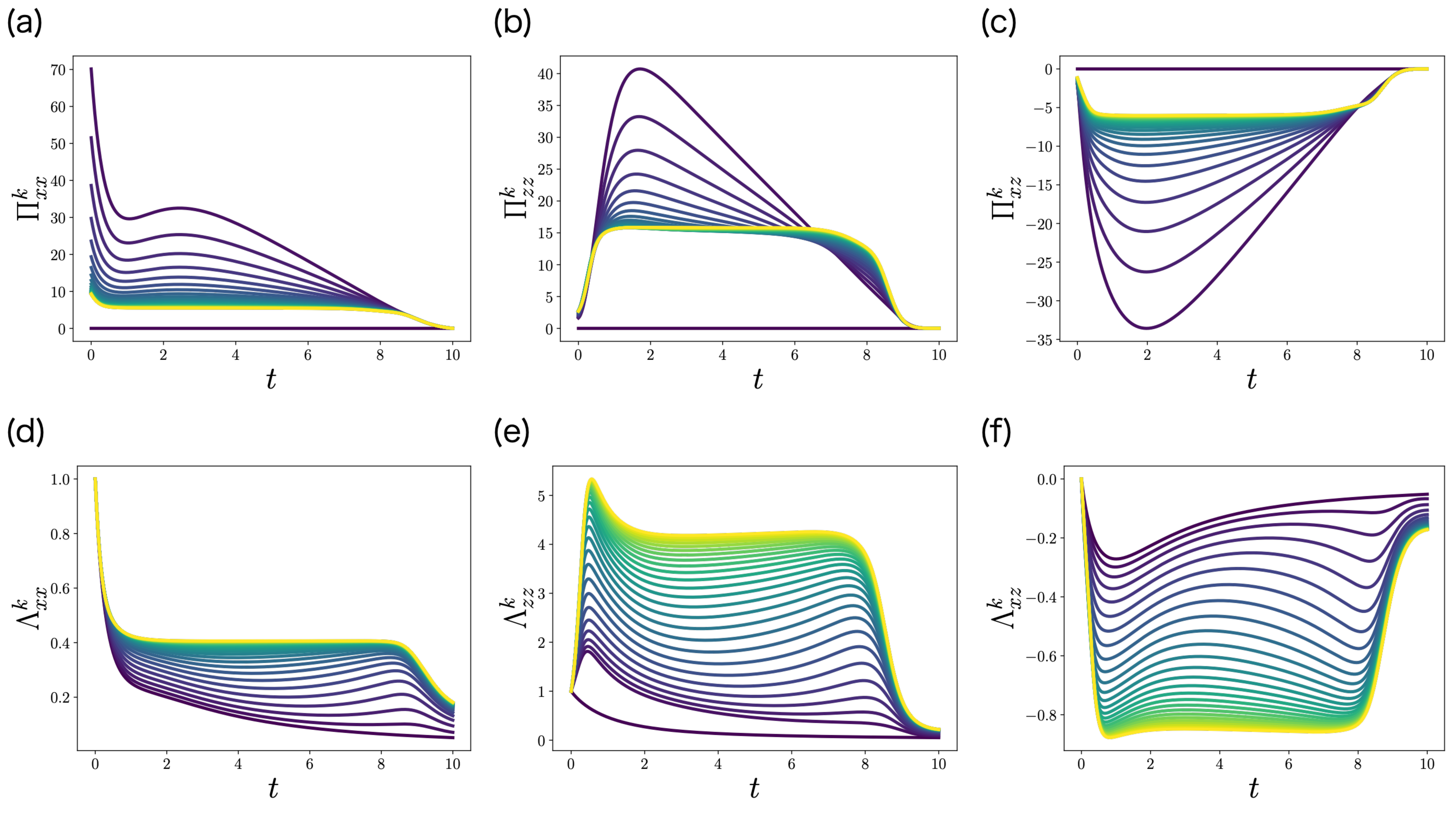}
	\caption{
	The elements of the control gain matrix $\Pi^{k}(t)\in\mb{R}^{2\times2}$ (a,b,c) and the precision matrix $\Lambda^{k}(t)\in\mb{R}^{2\times2}$ (d,e,f) obtained by FBSM (Algorithm \ref{alg: FBSM LQG}) in the numerical experiment of the LQG problem of ML-POSC. 
	Because $\Pi_{zx}^{k}(t)=\Pi_{xz}^{k}(t)$ and $\Lambda_{zx}^{k}(t)=\Lambda_{xz}^{k}(t)$, $\Pi_{zx}^{k}(t)$ and $\Lambda_{zx}^{k}(t)$ are not visualized. 
	The darkest curve corresponds to the first iteration $k=0$, and the brightest curve corresponds to the last iteration $k=50$. 
	$\Pi^{0}(t)$ is initialized by $\Pi^{0}(t)=O$. 
	}
	\label{fig: LQG Pi and Lambda}
\end{center}
\end{figure}

\begin{figure}[t]
\begin{center}
	\includegraphics[width=180mm]{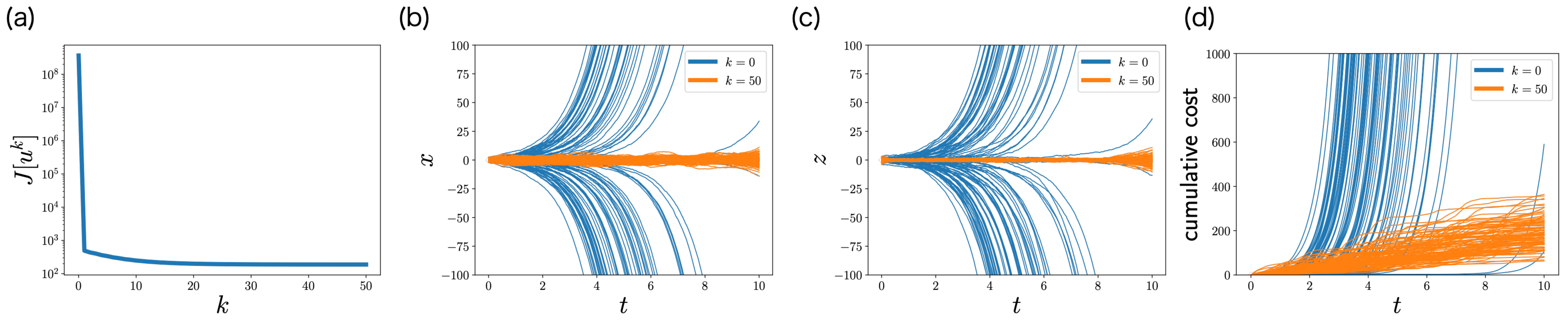}
	\caption{
	Performance of FBSM in the numerical experiment of the LQG problem of ML-POSC. 
	(a) The objective function $J[u^{k}]$ with respect to the iteration $k$. 
	(b,c,d) Stochastic simulation of the state $x_{t}$ (b), the memory $z_{t}$ (c), and the cumulative cost (d) for 100 samples. 
	The expectation of the cumulative cost at $t=10$ corresponds to the objective function (\ref{eq: OF of ML-POSC LQG NE}).
	Blue and orange curves correspond to the first iteration $k=0$ and the last iteration $k=50$, respectively. 
	}
	\label{fig: LQG performance}
\end{center}
\end{figure}

\subsection{Non-LQG Problem}
In this subsection, we verify the convergence of FBSM in ML-POSC by conducting the numerical experiment of the non-LQG problem. 
We consider the state $x_{t}\in\mb{R}$, the observation $y_{t}\in\mb{R}$, and the memory $z_{t}\in\mb{R}$, which evolve by the following SDEs: 
\begin{align}
	dx_{t}&=u_{t}dt+d\omega_{t},\label{eq: state SDE non-LQG NE}\\
	dy_{t}&=x_{t}dt+d\nu_{t},\label{eq: observation SDE non-LQG NE}\\
	dz_{t}&=dy_{t},\label{eq: memory SDE non-LQG NE}
\end{align}
where $x_{0}$ and $z_{0}$ obey the Gaussian distributions $p_{0}(x_{0})=\mcal{N}(x_{0}|0,0.01)$ and $p_{0}(z_{0})=\mcal{N}(z_{0}|0,0.01)$, respectively, $y_{0}$ is an arbitrary real number, $\omega_{t}\in\mb{R}$ and $\nu_{t}\in\mb{R}$ are independent standard Wiener processes, $u_{t}=u(t,z_{t})\in\mb{R}$ is the control. 
For the sake of simplicity, the memory control is not considered. 
The objective function to be minimized is given as follows: 
\begin{align}	
	J[u]:=\mb{E}_{p(x_{0:1},y_{0:1},z_{0:1};u)}\left[\int_{0}^{1}\left(Q(t,x_{t})+u_{t}^{2}\right)dt+10x_{1}^{2}\right], 
	\label{eq: OF of ML-POSC non-LQG NE}
\end{align} 
where  
\begin{align}
	Q(t,x):=\begin{cases}
		1000&(0.3\leq t\leq 0.6, 0.1\leq |x|\leq 2.0),\\
		0&(others).
	\end{cases}
	\label{eq: state cost function of non-LQG NE}
\end{align}
The cost function is high on the black rectangles in Figure \ref{fig: non-LQG performance} (b), which represents the obstacles. 
In addition, the terminal cost function is the lowest on the black cross in Figure \ref{fig: non-LQG performance} (b), which represents the desirable goal.
Therefore, the system should avoid the obstacles and reach the goal with the small control. 
Because the cost function is non-quadratic, it is a non-LQG problem. 

We apply FBSM (Algorithm \ref{alg: FBSM}) to this problem. 
$u^{0}(t,z)$ is initialized by $u^{0}(t,z)=0$. 
In order to solve the HJB equation and the FP equation, we use the finite-difference method. 
Figure \ref{fig: non-LQG HJB-FP} shows $w^{k}(t,s)$ and $p^{k}(t,s)$ obtained by FBSM. 
From the proof of Lemma \ref{lemm: current optimal control} (\ref{appendix: proof: lemm: current optimal control ver 1}), $w^{k}(t,s)$ is given as follows: 
\begin{align}
	w^{k}(t,s)=\mb{E}_{p(s_{t+dt:T}|s_{t}=s;u^{k})}\left[\int_{t}^{1}\left(Q(\tau,x_{\tau})+(u_{\tau}^{k})^{2}\right)d\tau+10x_{1}^{2}\right]. 
\end{align}
Because $u^{0}(t,z)=0$, $w^{0}(t,s)$ reflects the cost function corresponding to the obstacles and the goal (Figure \ref{fig: non-LQG HJB-FP}(a-e)). 
In contrast, because $u^{50}(t,z)\neq 0$, $w^{50}(t,s)$ becomes more complex (Figure \ref{fig: non-LQG HJB-FP}(f-j)). 
Especially, while $w^{0}(t,s)$ does not depends on the memory $z$, $w^{50}(t,s)$ depends on the memory $z$, which indicates that the control function $u^{50}(t,z)$ is adjusted by the memory $z$. 
Furthermore, while $p^{0}(t,s)$ is a unimodal distribution (Figure \ref{fig: non-LQG HJB-FP}(k-o)), $p^{50}(t,s)$ is a bimodal distribution (Figure \ref{fig: non-LQG HJB-FP}(p-t)), which can avoid the obstacles. 

Figure \ref{fig: non-LQG performance}(a) shows the objective function $J[u^{k}]$ with respect to the iteration $k$. 
The objective function $J[u^{k}]$ monotonically decreases with respect to the iteration $k$, which is consistent with Theorem \ref{theo: monotonicity of FBSM}. 
This monotonicity of FBSM is the nice property of ML-POSC that is not guaranteed in the deterministic control and MFSC. 
The objective function $J[u^{k}]$ finally converges, and its $u^{k}$ satisfies the Pontryagin's minimum principle from Theorem \ref{theo: convergence of FBSM}. 

Figure \ref{fig: non-LQG performance}(b,c) compares the performance of the control function $u^{k}$ at the first iteration $k=0$ and the last iteration $k=50$ by conducting the stochastic simulation. 
At the first iteration $k=0$, the obstacles cannot be avoided, which results in the higher objective function. 
In contrast, at the last iteration $k=50$, the obstacles can be avoided, which results in the lower objective function. 
This result indicates that FBSM improves the performance in ML-POSC. 

\begin{figure}[t]
\begin{center}
	\includegraphics[width=180mm]{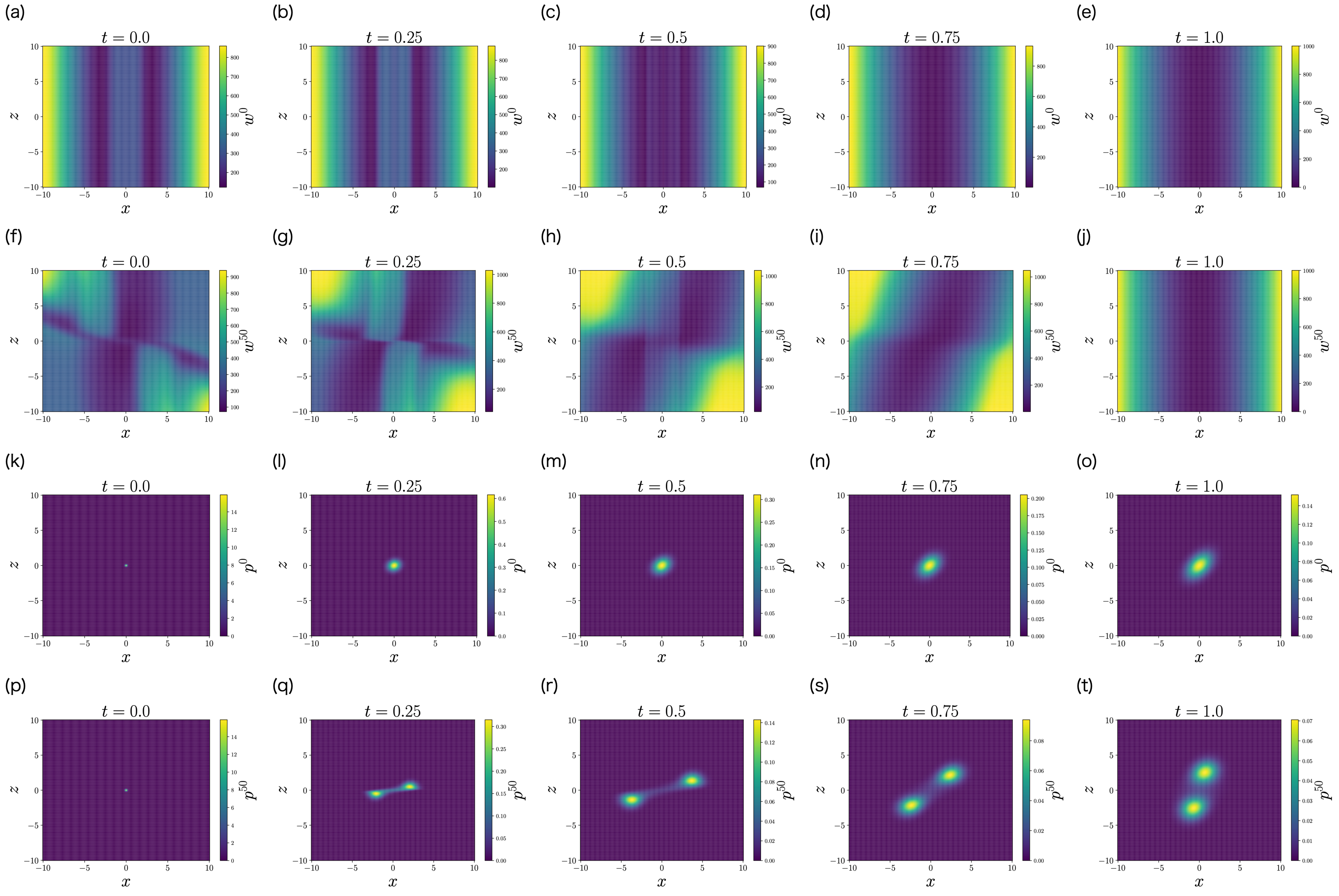}
	\caption{
	The solutions of the HJB equation $w^{k}(t,s)$ (a-j) and the FP equation $p^{k}(t,s)$ (k-t) at the first iteration $k=0$ (a-e,k-o) and at the last iteration $k=50$ (f-j,p-t) of FBSM (Algorithm \ref{alg: FBSM}) in the numerical experiment of the non-LQG problem of ML-POSC. 
	$u^{0}(t,z)$ is initialized by $u^{0}(t,z)=0$. 
	}
	\label{fig: non-LQG HJB-FP}
\end{center}
\end{figure}
\begin{figure}[t]
\begin{center}
	\includegraphics[width=135mm]{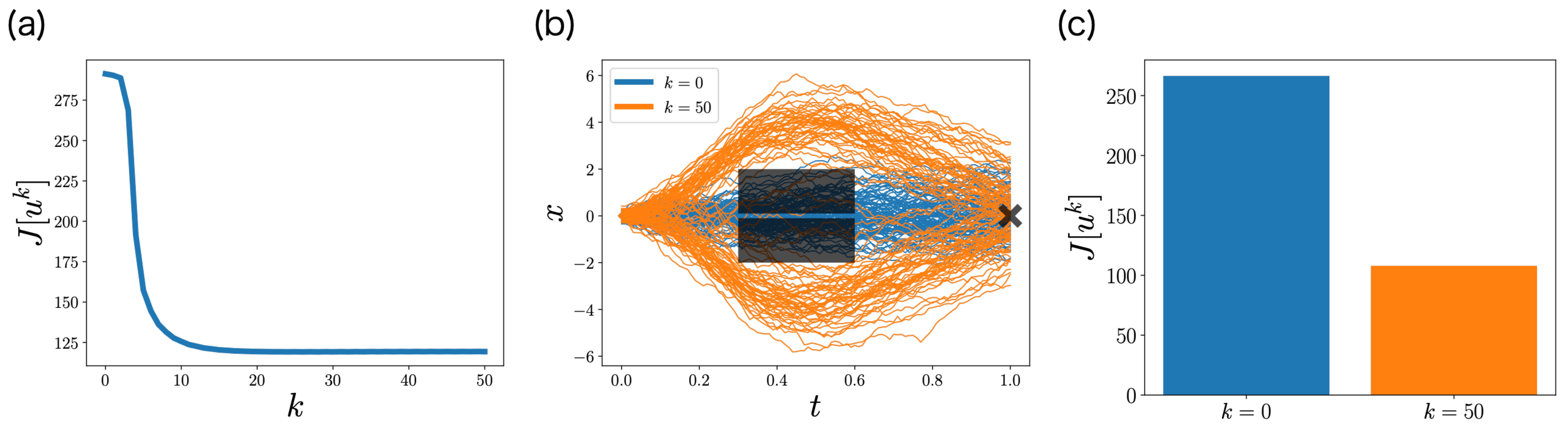}
	\caption{
	Performance of FBSM in the numerical experiment of the non-LQG problem of ML-POSC. 
	(a) The objective function $J[u^{k}]$ with respect to the iteration $k$. 
	(b) Stochastic simulation of the state $x_{t}$ for 100 samples. 
	The black rectangles and cross represent the obstacles and the goal, respectively. 
	Blue and orange curves correspond to the first iteration $k=0$ and the last iteration $k=50$, respectively. 
	(c) The objective function (\ref{eq: OF of ML-POSC non-LQG NE}), which is computed from 100 samples. 
	}
	\label{fig: non-LQG performance}
\end{center}
\end{figure}

\section{Discussion}\label{sec: conclusion}
In this work, we firstly showed that the system of HJB-FP equations corresponds to the Pontryagin's minimum principle on the probability density function space. 
Although the relationship between the system of HJB-FP equations and the Pontryagin's minimum principle has been mentioned briefly in MFSC \cite{crisan_master_2014,bensoussan_master_2015,bensoussan_interpretation_2017}, its details have not yet been investigated. 
We resolved this problem by deriving the system of HJB-FP equations in a similar way to the Pontryagin's minimum principle. 
We then proposed FBSM to ML-POSC. 
Although the convergence of FBSM is generally not guaranteed in the deterministic control \cite{krylov_method_1963,chernousko_method_1982,lenhart_optimal_2007,mcasey_convergence_2012} and MFSC \cite{carlini_semi-lagrangian_2013,carlini_fully_2014,carlini_semi-lagrangian_2015,lauriere_numerical_2021}, we proved the convergence in ML-POSC by noting the fact that the update of the current control function does not affect the future HJB equation in ML-POSC. 
Therefore, ML-POSC is a special and nice class where FBSM is guaranteed to converge. 

The regularized FBSM has recently been proposed in the deterministic control, which is guaranteed to converge even in the general deterministic control \cite{li_maximum_2018,liu_symplectic_2021}. 
Our work gives an intuitive reason why the regularized FBSM is guaranteed to converge. 
In the regularized FBSM, the Hamiltonian is regularized, which makes the update of the control function smaller. 
When the regularization is sufficiently strong, the effect from the current control function to the future backward dynamics is negligible. 
Therefore, the regularized FBSM of the deterministic control is guaranteed to converge for the similar reason to the FBSM of ML-POSC. 
However, the convergence of the regularized FBSM is much slower because the stronger regularization makes the update of the control function smaller. 
The FBSM of ML-POSC does not suffer from such a problem because the future backward dynamics already does not depend on the current control function without regularization. 

Our work gives a hint about a modification of the fixed-point iteration method to ensure the convergence in MFSC. 
Although the fixed-point iteration method is the most basic algorithm in MFSC, its convergence is not guaranteed \cite{carlini_semi-lagrangian_2013,carlini_fully_2014,carlini_semi-lagrangian_2015,lauriere_numerical_2021}.  
Our work showed that the fixed-point iteration method is equivalent with the FBSM on the probability density function space. 
Therefore, the idea of the regularized FBSM may also be applied to the fixed-point iteration method. 
More specifically, the fixed-point iteration method may be guaranteed to converge by regularizing the expected Hamiltonian. 

In FBSM, we solve the HJB equation and the FP equation by the finite-difference method. 
However, because the finite-difference method is prone to the curse of dimensionality, it is difficult to solve high-dimensional ML-POSC. 
In order to resolve this problem, we consider two directions. 
One direction is the policy iteration method \cite{puterman_markov_2014,bellman_dynamic_1957,howard_dynamic_1960}. 
Although the policy iteration method is almost the same with FBSM, only the update of the control function is different. 
While FBSM updates the system of HJB equations and the control function simultaneously, the policy iteration method updates them separately. 
In the policy iteration method, the system of HJB-FP equations becomes linear, which can be solved by the sampling method \cite{kappen_linear_2005,kappen_path_2005,satoh_iterative_2017}. 
Because the sampling method is more tractable than the finite-difference method, the policy iteration method may allow high-dimensional ML-POSC to be solved. 
Furthermore, the policy iteration method has recently been studied in MFSC \cite{cacace_policy_2021,lauriere_policy_2021,camilli_rates_2022}. 
However, its convergence is not guaranteed except for special cases in MFSC. 
In the similar way as FBSM, the convergence of the policy iteration method may be guaranteed in ML-POSC. 

The other direction is the machine learning. 
The neural network-based algorithms have recently been proposed in MFSC, which can solve high-dimensional problems efficiently \cite{ruthotto_machine_2020,lin_alternating_2021}. 
By extending these algorithms, high-dimensional ML-POSC may be solved efficiently. 
Furthermore, unlike MFSC, the coupling of HJB-FP equations is limited only to the optimal control function in ML-POSC. 
By exploiting this nice property, more efficient algorithms may be devised for ML-POSC.

\setcounter{section}{0}
\renewcommand{\thesection}{Appendix \Roman{section}}
\section{Deterministic Control}\label{appendix: DC}
In this section, we review the Pontryagin's minimum principle in the deterministic control \cite{yong_stochastic_1999,vinter_optimal_2010,lewis_optimal_2012,aschepkov_optimal_2016}. 

\subsection{Problem Formulation}\label{appendix: DC-Problem formulation}
In this subsection, we formulate the deterministic control \cite{vinter_optimal_2010,lewis_optimal_2012,aschepkov_optimal_2016}. 
The state of the system $s_{t}\in\mb{R}^{d_{s}}$ at time $t\in[0,T]$ evolves by the following ordinary differential equation (ODE): 
\begin{align}
	\frac{ds_{t}}{dt}=b(t,s_{t},u_{t}), 
	\label{eq: state equation}
\end{align}
where the initial state is $s_{0}$, and the control is $u_{t}=u(t)\in\mb{R}^{d_{u}}$. 
The objective function is given by the following cumulative cost function: 
\begin{align}	
	J[u]:=\int_{0}^{T}f(t,s_{t},u_{t})dt+g(s_{T}),
	\label{eq: OF of DC}
\end{align}
where $f$ is the cost function, and $g$ is the terminal cost function.  
The deterministic control is the problem to find the optimal control function $u^{*}$ that minimizes the cumulative cost function $J[u]$ as follows: 
\begin{align}	
	u^{*}=\argmin_{u}J[u].
	\label{eq: DC}
\end{align}

\subsection{Hamiltonian}\label{appendix: DC-Hamiltonian}
Before we show the optimality conditions, we define the Hamiltonian as follows: 
\begin{align}
	\mcal{H}\left(t,s,u,\lambda\right):=f(t,s,u)+\lambda^{\top}b(t,s,u), 
	\label{eq: Hamiltonian DC}
\end{align}
where $^\forall \lambda\in\mb{R}^{d_{s}}$ is the adjoint variable. 
The following lemma shows the relationship between the objective function $J[u]$ and the Hamiltonian $\mcal{H}(t,s,u,\lambda)$, 
which is significant for the optimality conditions \cite{yong_stochastic_1999,vinter_optimal_2010}. 

\begin{lemm}[\cite{yong_stochastic_1999,vinter_optimal_2010}]\label{lemm: J-J DC}
Let $^\forall u: [0,T]\to\mb{R}^{d_{u}}$ and $^\forall u': [0,T]\to\mb{R}^{d_{u}}$ be the arbitrary control functions, and let $s$ and $s'$ be the states driven by the control functions $u$ and $u'$, respectively.  
Then $J[u]-J[u']$ satisfies the following equation: 
\begin{align}
	J[u]-J[u']&=\int_{0}^{T}\left(\mcal{H}(t,s_{t},u_{t},\lambda_{t}')-\mcal{H}(t,s_{t}',u_{t}',\lambda_{t}')-\left(\frac{\pl \mcal{H}(t,s_{t}',u_{t}',\lambda_{t}')}{\pl s}\right)^{\top}(s_{t}-s_{t}')\right)dt\nonumber\\
	&\ \ \ +g(s_{T})-g(s_{T}')-\left(\frac{\pl g(s_{T}')}{\pl s}\right)^{\top}(s_{T}-s_{T}'),
	\label{eq: J-J DC}
\end{align}
where $\lambda_{t}'$ is the adjoint variable of $s_{t}'$, which is the solution of the following ODE: 
\begin{align}
	-\frac{d\lambda_{t}'}{dt}=\frac{\pl \mcal{H}\left(t,s_{t}',u_{t}',\lambda_{t}'\right)}{\pl s}, 
	\label{eq: adjoint equation}
\end{align}
where $\lambda_{T}'=\pl g(s_{T}')/\pl s$. 
\end{lemm}

\begin{proof}
$J[u]-J[u']$ can be calculated as follows: 
\begin{align}
	J[u]-J[u']
	&=\left[\int_{0}^{T}f(t,s_{t},u_{t})dt+g(s_{T})\right]-\left[\int_{0}^{T}f(t,s_{t}',u_{t}')dt+g(s_{T}')\right]\nonumber\\
	&=\left[\int_{0}^{T}\left(\mcal{H}(t,s_{t},u_{t},\lambda_{t}')-(\lambda_{t}')^{\top}b(t,s_{t},u_{t})\right)dt+g(s_{T})\right]\nonumber\\
	&\ \ \ -\left[\int_{0}^{T}\left(\mcal{H}(t,s_{t}',u_{t}',\lambda_{t}')-(\lambda_{t}')^{\top}b(t,s_{t}',u_{t}')\right)dt+g(s_{T}')\right]\nonumber\\
	&=\int_{0}^{T}\left(\mcal{H}(t,s_{t},u_{t},\lambda_{t}')-\mcal{H}(t,s_{t}',u_{t}',\lambda_{t}')\right)dt\nonumber\\
	&\ \ \ -\int_{0}^{T}(\lambda_{t}')^{\top}\left(b(t,s_{t},u_{t})-b(t,s_{t}',u_{t}')\right)dt+g(s_{T})-g(s_{T}').
\end{align}
From the state equation (\ref{eq: state equation}), 
\begin{align}
	J[u]-J[u']
	&=\int_{0}^{T}\left(\mcal{H}(t,s_{t},u_{t},\lambda_{t}')-\mcal{H}(t,s_{t}',u_{t}',\lambda_{t}')\right)dt\nonumber\\
	&\ \ \ -\int_{0}^{T}(\lambda_{t}')^{\top}\frac{d(s_{t}-s_{t}')}{dt}dt+g(s_{T})-g(s_{T}').
\end{align}
From the integration by parts and $s_{0}-s_{0}'=0$, 
\begin{align}
	J[u]-J[u']
	&=\int_{0}^{T}\left(\mcal{H}(t,s_{t},u_{t},\lambda_{t}')-\mcal{H}(t,s_{t}',u_{t}',\lambda_{t}')\right)dt\nonumber\\
	&\ \ \ +\int_{0}^{T}\left(\frac{d\lambda_{t}'}{dt}\right)^{\top}(s_{t}-s_{t}')dt+g(s_{T})-g(s_{T}')-(\lambda_{T}')^{\top}(s_{T}-s_{T}').
\end{align}
From the adjoint equation (\ref{eq: adjoint equation}), (\ref{eq: J-J DC}) is obtained. 
\end{proof}

\subsection{Necessary Condition}\label{appendix: DC-NC}
We show the Pontryagin's minimum principle, which is a necessary condition of the optimal control function as follows \cite{yong_stochastic_1999,vinter_optimal_2010}. 
\begin{prop}[\cite{yong_stochastic_1999,vinter_optimal_2010}]\label{prop: NC DC}
In the deterministic control, the optimal control function $u^{*}$ satisfies the following equation: 
\begin{align}
	u^{*}(t)=\argmin_{u}\mcal{H}\left(t,s_{t}^{*},u,\lambda_{t}^{*}\right),\ ^\forall t\in[0,T],
	\label{eq: optimal control of DC}
\end{align}
where $s_{t}^{*}$ is the state driven by the optimal control function $u^{*}$, which is the solution of the following state equation: 
\begin{align}
	\frac{ds_{t}}{dt}=\frac{\pl \mcal{H}(t,s_{t}^{*},u_{t}^{*},\lambda_{t}^{*})}{\pl \lambda},
	\label{eq: optimal state equation}
\end{align}
where the initial state is $s_{0}$. 
Because $\pl \mcal{H}(t,s_{t}^{*},u_{t}^{*},\lambda_{t}^{*})/\pl \lambda=b(t,s_{t}^{*},u_{t}^{*})$, (\ref{eq: optimal state equation}) is consistent with (\ref{eq: state equation}). 
$\lambda_{t}^{*}$ is the adjoint variable of $s_{t}^{*}$, which is the solution of the following adjoint equation: 
\begin{align}
	-\frac{d\lambda_{t}^{*}}{dt}=\frac{\pl \mcal{H}(t,s_{t}^{*},u_{t}^{*},\lambda_{t}^{*})}{\pl s}, 
	\label{eq: optimal adjoint equation}
\end{align}
where $\lambda_{T}^{*}=\pl g(s_{T}^{*})/\pl s$. 
\end{prop}

\begin{proof}
We define the control function
\begin{align}
	u^{\ve}(t):=\begin{cases}
		u^{*}(t)&t\in[0,T]\backslash E_{\ve},\\
		u(t)&t\in E_{\ve},\\
	\end{cases}
\end{align}
where $E_{\ve}:=[t',t'+\ve]\subseteq[0,T]$, and $^\forall u:[0,T]\to\mb{R}^{d_{u}}$. 
From the Lemma \ref{lemm: J-J DC}, $J[u^{\ve}]-J[u^{*}]$ can be calculated as follows: 
\begin{align}
	J[u^{\ve}]-J[u^{*}]&=\int_{0}^{T}\left(\mcal{H}(t,s_{t}^{\ve},u_{t}^{\ve},\lambda_{t}^{*})-\mcal{H}(t,s_{t}^{*},u_{t}^{*},\lambda_{t}^{*})-\left(\frac{\pl \mcal{H}(t,s_{t}^{*},u_{t}^{*},\lambda_{t}^{*})}{\pl s}\right)^{\top}(s_{t}^{\ve}-s_{t}^{*})\right)dt\nonumber\\
	&\ \ \ +g(s_{T}^{\ve})-g(s_{T}^{*})-\left(\frac{\pl g(s_{T}^{*})}{\pl s}\right)^{\top}(s_{T}^{\ve}-s_{T}^{*})\nonumber\\
	&=\int_{0}^{T}\left(\mcal{H}(t,s_{t}^{\ve},u_{t}^{*},\lambda_{t}^{*})-\mcal{H}(t,s_{t}^{*},u_{t}^{*},\lambda_{t}^{*})-\left(\frac{\pl \mcal{H}(t,s_{t}^{*},u_{t}^{*},\lambda_{t}^{*})}{\pl s}\right)^{\top}(s_{t}^{\ve}-s_{t}^{*})\right)dt\nonumber\\
	&\ \ \ +g(s_{T}^{\ve})-g(s_{T}^{*})-\left(\frac{\pl g(s_{T}^{*})}{\pl s}\right)^{\top}(s_{T}^{\ve}-s_{T}^{*})\nonumber\\
	&\ \ \ +\int_{E_{\ve}}\left(\mcal{H}(t,s_{t}^{\ve},u_{t},\lambda_{t}^{*})-\mcal{H}(t,s_{t}^{\ve},u_{t}^{*},\lambda_{t}^{*})\right)dt.
\end{align}
Letting $\ve\to0$,
\begin{align}
	J[u^{\ve}]-J[u^{*}]&=\int_{0}^{T}\left(\left(\frac{\pl \mcal{H}(t,s_{t}^{*},u_{t}^{*},\lambda_{t}^{*})}{\pl s}\right)^{\top}(s_{t}^{\ve}-s_{t}^{*})-\left(\frac{\pl \mcal{H}(t,s_{t}^{*},u_{t}^{*},\lambda_{t}^{*})}{\pl s}\right)^{\top}(s_{t}^{\ve}-s_{t}^{*})\right)dt\nonumber\\
	&\ \ \ +\left(\frac{\pl g(s_{T}^{*})}{\pl s}\right)^{\top}(s_{T}^{\ve}-s_{T}^{*})-\left(\frac{\pl g(s_{T}^{*})}{\pl s}\right)^{\top}(s_{T}^{\ve}-s_{T}^{*})\nonumber\\
	&\ \ \ +\left(\mcal{H}(t',s_{t'}^{*},u_{t'},\lambda_{t'}^{*})-\mcal{H}(t',s_{t'}^{*},u_{t'}^{*},\lambda_{t'}^{*})\right)dt\nonumber\\
	&=\left(\mcal{H}(t',s_{t'}^{*},u_{t'},\lambda_{t'}^{*})-\mcal{H}(t',s_{t'}^{*},u_{t'}^{*},\lambda_{t'}^{*})\right)dt.
\end{align}
Because $u^{*}$ is the optimal control function, the following inequality is satisfied: 
\begin{align}
	0\leq J[u^{\ve}]-J[u^{*}]
	&=\left(\mcal{H}(t',s_{t'}^{*},u_{t'},\lambda_{t'}^{*})-\mcal{H}(t',s_{t'}^{*},u_{t'}^{*},\lambda_{t'}^{*})\right)dt.
\end{align}
Therefore, (\ref{eq: optimal control of DC}) is obtained. 
\end{proof}

\subsection{Sufficient Condition}\label{appendix: DC-SC}
The Pontryagin's minimum principle becomes a necessary and sufficient condition if the Hamiltonian and the terminal cost function are convex as follows \cite{yong_stochastic_1999,vinter_optimal_2010}: 

\begin{prop}[\cite{yong_stochastic_1999,vinter_optimal_2010}]\label{prop: SC DC}
Assume that $\mcal{H}(t,s,u,\lambda)$ is convex with respect to $s$ and $u$ and $g(s)$ is convex with respect to $s$. 
If the control function $u^{*}$ satisfies (\ref{eq: optimal control of DC}), then it is the optimal control function of the deterministic control. 
\end{prop}

\begin{proof}
We define the arbitrary control function $^\forall u: [0,T]\to\mb{R}^{d_{u}}$. 
From the Lemma \ref{lemm: J-J DC}, $J[u]-J[u^{*}]$ is given by the following equation: 
\begin{align}
	J[u]-J[u^{*}]&=\int_{0}^{T}\left(\mcal{H}(t,s_{t},u_{t},\lambda_{t}^{*})-\mcal{H}(t,s_{t}^{*},u_{t}^{*},\lambda_{t}^{*})-\left(\frac{\pl \mcal{H}(t,s_{t}^{*},u_{t}^{*},\lambda_{t}^{*})}{\pl s}\right)^{\top}(s_{t}-s_{t}^{*})\right)dt\nonumber\\
	&\ \ \ +g(s_{T})-g(s_{T}^{*})-\left(\frac{\pl g(s_{T}^{*})}{\pl s}\right)^{\top}(s_{T}-s_{T}^{*}).
\end{align}
Because $\mcal{H}(t,s,u,\lambda)$ is convex with respect to $s$ and $u$ and $g(s)$ is convex with respect to $s$, the following inequalities are satisfied: 
\begin{align}
	\mcal{H}(t,s_{t},u_{t},\lambda_{t}^{*})&\geq \mcal{H}(t,s_{t}^{*},u_{t}^{*},\lambda_{t}^{*})+\left(\frac{\pl \mcal{H}(t,s_{t}^{*},u_{t}^{*},\lambda_{t}^{*})}{\pl s}\right)^{\top}(s_{t}-s_{t}^{*})\nonumber\\
	&\ \ \ +\left(\frac{\pl \mcal{H}(t,s_{t}^{*},u_{t}^{*},\lambda_{t}^{*})}{\pl u}\right)^{\top}(u_{t}-u_{t}^{*}),\\
	g(s_{T})&\geq g(s_{T}^{*})+\left(\frac{\pl g(s_{T}^{*})}{\pl s}\right)^{\top}(s_{T}-s_{T}^{*}).
\end{align}
Hence, the following inequality is satisfied: 
\begin{align}
	J[u]-J[u^{*}]
	&\geq\int_{0}^{T}\left(\frac{\pl \mcal{H}(t,s_{t}^{*},u_{t}^{*},\lambda_{t}^{*})}{\pl u}\right)^{\top}(u_{t}-u_{t}^{*})dt. 
\end{align}
Because $u^{*}$ satisfies (\ref{eq: optimal control of DC}), the following stationary condition is satisfied: 
\begin{align}
	\frac{\pl \mcal{H}(t,s_{t}^{*},u_{t}^{*},\lambda_{t}^{*})}{\pl u}=0. 
\end{align}
Hence, the following inequality is satisfied: 
\begin{align}
	J[u]-J[u^{*}]&\geq0. 
\end{align}
Therefore, $u^{*}$ is the optimal control function.
\end{proof}

\subsection{Relationship with Bellman's Dynamic Programming Principle}\label{appendix: DC-DPP}
From the Bellman's dynamic programming principle, the optimal control function $u^{*}$ is given as follows \cite{yong_stochastic_1999,vinter_optimal_2010}: 
\begin{align}
	u^{*}(t,s)=\argmin_{u}\mcal{H}\left(t,s,u,\frac{\pl V^{*}(t,s)}{\pl s}\right). 
	\label{eq: optimal control of DC DPP}
\end{align}
More specifically, the optimal control function is given by $u^{*}(t)=u^{*}(t,s_{t}^{*})$. 
$V^{*}(t,s)$ is the value function, which is the solution of the following Hamilton-Jacobi-Bellman (HJB) equation: 
\begin{align}
	-\frac{\pl V^{*}(t,s)}{\pl t}=\mcal{H}\left(t,s,u^{*},\frac{\pl V^{*}(t,s)}{\pl s}\right),
	\label{eq: HJB eq DC}
\end{align}
where $V^{*}(T,s)=g(s)$. 
The following proposition converts the Bellman's dynamic programming principle into the Pontryagin's minimum principle \cite{yong_stochastic_1999,vinter_optimal_2010}: 

\begin{prop}[\cite{yong_stochastic_1999,vinter_optimal_2010}]\label{prop: DPP DC}
We define $\lambda_{t}^{*}$ from $V^{*}(t,s)$ as follows: 
\begin{align}
	\lambda_{t}^{*}:=\frac{\pl V^{*}(t,s_{t}^{*})}{\pl s}, 
\end{align}
where $s_{t}^{*}$ is the solution of the state equation (\ref{eq: optimal state equation}). 
$\lambda_{t}^{*}$ satisfies the adjoint equation (\ref{eq: optimal adjoint equation}) from the HJB equation (\ref{eq: HJB eq DC}). 
\end{prop}

\begin{proof}
We firstly define 
\begin{align}
	\Lambda^{*}(t,s):=\frac{\pl V^{*}(t,s)}{\pl s}. 
\end{align}
Differentiating the HJB equation (\ref{eq: HJB eq DC}) with respect to $s$, 
the following equation is obtained: 
\begin{align}
	-\frac{\pl \Lambda^{*}(t,s)}{\pl t}=\frac{\pl \mcal{H}\left(t,s,u^{*},\Lambda^{*}\right)}{\pl s}
	+\left(\frac{\pl \Lambda^{*}(t,s)}{\pl s}\right)^{\top}b(t,s^{*},u^{*}),
	\label{eq: Master eq DC}
\end{align}
where $\Lambda^{*}(T,s)=\pl g(s)/\pl s$. 
Then the derivative of $\lambda_{t}^{*}=\Lambda^{*}(t,s_{t}^{*})$ with respect to $t$ can be calculated as follows: 
\begin{align}
	\frac{d \lambda_{t}^{*}}{d t}=\frac{\pl\Lambda^{*}(t,s_{t}^{*})}{\pl t}+\left(\frac{\pl \Lambda^{*}(t,s_{t}^{*})}{\pl s}\right)^{\top}\frac{ds_{t}^{*}}{dt}. 
	\label{eq: the derivative of lambda(t)}
\end{align}
By substituting (\ref{eq: Master eq DC}) into (\ref{eq: the derivative of lambda(t)}), the following equation is obtained: 
\begin{align}
	-\frac{d \lambda_{t}^{*}}{d t}=\frac{\pl \mcal{H}\left(t,s,u^{*},\lambda^{*}\right)}{\pl s}
	-\left(\frac{\pl \Lambda^{*}(t,s_{t}^{*})}{\pl s}\right)^{\top}\underbrace{\left(\frac{ds_{t}^{*}}{dt}-b(t,s^{*},u^{*})\right)}_{(*)}. 
	\label{eq: the derivative of lambda(t) ver2}
\end{align}
From the state equation (\ref{eq: optimal state equation}), $(*)=0$ is satisfied. 
Therefore, $\lambda^{*}(t)$ satisfies the adjoint equation (\ref{eq: optimal adjoint equation}).
\end{proof}

\section{Mean-Field Stochastic Control}\label{appendix: MFC}
In this section, we show that the system of HJB-FP equations  in MFSC corresponds to the Pontryagin's minimum principle on the probability density function space.  
Although the relationship between the system of HJB-FP equations and the Pontryagin's minimum principle has been mentioned briefly in MFSC \cite{crisan_master_2014,bensoussan_master_2015,bensoussan_interpretation_2017}, its details have not yet been investigated. 
In this section, we resolve this problem by deriving the system of HJB-FP equations in a similar way as \ref{appendix: DC}. 

\subsection{Problem Formulation}\label{appendix: MFC-Problem formulation}
In this subsection, we formulate MFSC \cite{bensoussan_mean_2013,carmona_probabilistic_2018,carmona_probabilistic_2018-1}. 
The state of the system $s_{t}\in\mb{R}^{d_{s}}$ at time $t\in[0,T]$ evolves by the following stochastic differential equation (SDE): 
\begin{align}
	ds_{t}=b(t,s_{t},p_{t},u_{t})dt+\sigma(t,s_{t},p_{t},u_{t})d\omega_{t},
	\label{eq: state SDE MFC}
\end{align}
where $s_{0}$ obeys $p_{0}(s_{0})$, $p_{t}(s):=p(t,s)$ is the probability density function of the state $s$, $u_{t}(s):=u(t,s)\in\mb{R}^{d_{u}}$ is the control, and $\omega_{t}\in\mb{R}^{d_{\omega}}$ is the standard Wiener process. 
The objective function is given by the following expected cumulative cost function: 
\begin{align}	
	J[u]:=\mb{E}_{p(s_{0:T};u)}\left[\int_{0}^{T}f(t,s_{t},p_{t},u_{t})dt+g(s_{T},p_{T})\right],
	\label{eq: OF of MFC}
\end{align}
where $f$ is the cost function, $g$ is the terminal cost function, $p(s_{0:T};u)$ is the probability of $s_{0:t}:=\{s_{\tau}|\tau\in[0,t]\}$ given $u$ as a parameter, and $\mb{E}_{p}[\cdot]$ is the expectation with respect to the probability $p$. 
MFSC is the problem to find the optimal control function $u^{*}$ that minimizes the expected cumulative cost function $J[u]$ as follows: 
\begin{align}	
	u^{*}=\argmin_{u}J[u].
	\label{eq: MFC}
\end{align}

\subsection{Hamiltonian}\label{appendix: MFC-Hamiltonian}
Before we show the optimality conditions, we define the Hamiltonian as follows: 
\begin{align}
	\mcal{H}\left(t,s,p,u,w\right):=f(t,s,p,u)+\mcal{L}_{u}w(t,s), 
	\label{eq: Hamiltonian MFC} 
\end{align}
where $^\forall w:[0,T]\times\mb{R}^{d_{s}}\to\mb{R}$, and $\mcal{L}_{u}$ is the backward diffusion operator, which is defined as follows: 
\begin{align}
	\mcal{L}_{u}w(t,s)&:=\sum_{i=1}^{d_{s}}b_{i}(t,s,p,u)\frac{\pl w(t,s)}{\pl s_{i}}+\frac{1}{2}\sum_{i,j=1}^{d_{s}}D_{ij}(t,s,p,u)\frac{\pl^{2} w(t,s)}{\pl s_{i}\pl s_{j}}, 
	\label{eq: backward diffusion operator MFC}
\end{align}
where $D(t,s,p,u):=\sigma(t,s,p,u)\sigma^{\top}(t,s,p,u)$. 
We also define the expected Hamiltonian and the expected terminal cost function as follows: 
\begin{align}
	\bar{\mcal{H}}(t,p,u,w)&:=\mb{E}_{p(s)}\left[\mcal{H}\left(t,s,p,u,w\right)\right],\label{eq: expected Hamiltonian MFC}\\
	\bar{g}(p)&:=\mb{E}_{p(s)}\left[g(s,p)\right].\label{eq: expected terminal cost MFC}
\end{align}
The state SDE (\ref{eq: state SDE MFC}) can be converted into the following Fokker-Planck (FP) equation: 
\begin{align}
	\frac{\pl p(t,s)}{\pl t}=\mcal{L}_{u}^{\dagger}p(t,s),
	\label{eq: FP eq MFC}
\end{align}
where $p(0,s)=p_{0}(s)$, and $\mcal{L}_{u}^{\dag}$ is the forward diffusion operator, which is defined as follows: 
\begin{align}
	\mcal{L}_{u}^{\dag}p(t,s)&:=-\sum_{i=1}^{d_{s}}\frac{\pl (b_{i}(t,s,p,u)p(t,s))}{\pl s_{i}}+\frac{1}{2}\sum_{i,j=1}^{d_{s}}\frac{\pl^{2}(D_{ij}(t,s,p,u)p(t,s))}{\pl s_{i}\pl s_{j}}. 
	\label{eq: forward diffusion operator MFC}
\end{align}
We note that $\mcal{L}_{u}^{\dag}$ is the conjugate of $\mcal{L}_{u}$ as follows: 
\begin{align}
	\int w(t,s)\mcal{L}_{u}^{\dag}p(t,s)ds=\int p(t,s)\mcal{L}_{u}w(t,s)ds.
	\label{eq: conjugate MFC}
\end{align}
The following lemma shows the relationship between the objective function $J[u]$ and the expected Hamiltonian $\bar{\mcal{H}}(t,p,u,w)$, 
which is significant for the optimality conditions. 

\begin{lemm}\label{lemm: J-J MFC}
Let $^\forall u: [0,T]\times\mb{R}^{d_{s}}\to\mb{R}^{d_{u}}$ and $^\forall u': [0,T]\times\mb{R}^{d_{s}}\to\mb{R}^{d_{u}}$ be the arbitrary control functions, and let $p$ and $p'$ be the probability density functions of the state driven by the control functions $u$ and $u'$, respectively.  
Then $J[u]-J[u']$ satisfies the following equation: 
\begin{align}
	J[u]-J[u']
	&=\int_{0}^{T}\left(\bar{\mcal{H}}(t,p,u,w')-\bar{\mcal{H}}(t,p',u',w')-\int\frac{\delta \bar{\mcal{H}}\left(t,p',u',w'\right)}{\delta p}(s)\left(p(t,s)-p'(t,s)\right)ds\right)dt\nonumber\\
	&\ \ \ +\bar{g}(p)-\bar{g}(p')
	-\int\frac{\delta \bar{g}(p')}{\delta p}(s)\left(p(T,s)-p'(T,s)\right)ds,
	\label{eq: J-J MFC}
\end{align}
where $w'$ is the solution of the following Hamilton-Jacobi-Bellman (HJB) equation: 
\begin{align}
	-\frac{\pl w'(t,s)}{\pl t}=\frac{\delta\bar{\mcal{H}}\left(t,p',u',w'\right)}{\delta p}(s), 
	\label{eq: HJB eq MFC}
\end{align}
where $w'(T,s)=(\delta\bar{g}(p')/\delta p)(s)$. 
\end{lemm}

\begin{proof}
$J[u]-J[u']$ can be calculated as follows: 
\begin{align}
	J[u]-J[u']
	&=\mb{E}_{p(s_{0:T})}\left[\int_{0}^{T}f(t,s_{t},p_{t},u_{t})dt+g(s_{T},p_{T})\right]\nonumber\\
	&\ \ \ -\mb{E}_{p'(s_{0:T})}\left[\int_{0}^{T}f(t,s_{t},p_{t}',u_{t}')dt+g(s_{T},p_{T}')\right]\nonumber\\
	&=\mb{E}_{p(s_{0:T})}\left[\int_{0}^{T}\left(\mcal{H}(t,s_{t},p_{t},u_{t},w')-\mcal{L}_{u_{t}}w'(t,s_{t})\right)dt+g(s_{T},p_{T})\right]\nonumber\\
	&\ \ \ -\mb{E}_{p'(s_{0:T})}\left[\int_{0}^{T}\left(\mcal{H}(t,s_{t},p_{t}',u_{t}',w')-\mcal{L}_{u_{t}'}w'(t,s_{t})\right)dt+g(s_{T},p_{T}')\right]\nonumber\\
	&=\int_{0}^{T}\left(\bar{\mcal{H}}(t,p,u,w')-\bar{\mcal{H}}(t,p',u',w')\right)dt\nonumber\\
	&\ \ \ -\int_{0}^{T}\left(\mb{E}_{p(t,s)}\left[\mcal{L}_{u}w'(t,s)\right]-\mb{E}_{p'(t,s)}\left[\mcal{L}_{u'}w'(t,s)\right]\right)dt+\bar{g}(p)-\bar{g}(p').
\end{align}
Because $\mcal{L}_{u_{t}}$ and $\mcal{L}_{u_{t}'}$ are the conjugates of $\mcal{L}_{u_{t}}^{\dag}$ and $\mcal{L}_{u_{t}'}^{\dag}$, respectively, 
\begin{align}
	J[u]-J[u']
	&=\int_{0}^{T}\left(\bar{\mcal{H}}(t,p,u,w')-\bar{\mcal{H}}(t,p',u',w')\right)dt\nonumber\\
	&\ \ \ -\int_{0}^{T}\int \left(\mcal{L}_{u}^{\dag}p(t,s)-\mcal{L}_{u'}^{\dag}p'(t,s)\right)w'(t,s)dsdt+\bar{g}(p)-\bar{g}(p').
\end{align}
From the FP equation (\ref{eq: FP eq MFC}), 
\begin{align}
	J[u]-J[u']
	&=\int_{0}^{T}\left(\bar{\mcal{H}}(t,p,u,w')-\bar{\mcal{H}}(t,p',u',w')\right)dt\nonumber\\
	&\ \ \ -\int_{0}^{T}\int \frac{\pl\left(p(t,s)-p'(t,s)\right)}{\pl t}w'(t,s)dsdt+\bar{g}(p)-\bar{g}(p').
\end{align}
From the integration by parts and $p(0,s)-p'(0,s)=p_{0}(s)-p_{0}(s)=0$, 
\begin{align}
	J[u]-J[u']
	&=\int_{0}^{T}\left(\bar{\mcal{H}}(t,p,u,w')-\bar{\mcal{H}}(t,p',u',w')\right)dt\nonumber\\
	&\ \ \ +\int_{0}^{T}\int \left(p(t,s)-p'(t,s)\right)\frac{\pl w'(t,s)}{\pl t}dsdt\nonumber\\
	&\ \ \ +\bar{g}(p)-\bar{g}(p')-\int \left(p(T,s)-p'(T,s)\right)w'(T,s)ds.
\end{align}
From the HJB equation (\ref{eq: HJB eq MFC}),  (\ref{eq: J-J MFC}) is obtained. 
\end{proof}

\subsection{Necessary Condition}\label{appendix: MFC-NC}
We show the necessary condition of the optimal control function of MFSC that corresponds to the Pontryagin's minimum principle on the probability density function space. 

\begin{theo}\label{theo: NC MFC}
In MFSC, the optimal control function $u^{*}$ satisfies the following equation: 
\begin{align}
	u^{*}(t,s)=\argmin_{u}\mcal{H}\left(t,s,p^{*},u,w^{*}\right),\ a.s.\ ^\forall t\in[0,T],\ ^\forall s\in\mb{R}^{d_{s}},
	\label{eq: optimal control of MFC}
\end{align}
where $p^{*}(t,s)$ is the solution of the following FP equation: 
\begin{align}
	\frac{\pl p^{*}(t,s)}{\pl t}=\frac{\delta\bar{\mcal{H}}(t,p^{*},u^{*},w^{*})}{\delta w}(s),
	\label{eq: optimal FP eq MFC}
\end{align}
where $p^{*}(0,s)=p_{0}(s)$. 
Because $(\delta\bar{\mcal{H}}(t,p^{*},u^{*},w^{*})/\delta w)(s)=\mcal{L}_{u^{*}}^{\dagger}p^{*}(t,s)$, (\ref{eq: optimal FP eq MFC}) is consistent with (\ref{eq: FP eq MFC}). 
$w^{*}(t,s)$ is the solution of the following HJB equation: 
\begin{align}
	-\frac{\pl w^{*}(t,s)}{\pl t}=\frac{\delta \bar{\mcal{H}}(t,p^{*},u^{*},w^{*})}{\delta p}(s), 
	\label{eq: optimal HJB eq MFC}
\end{align}
where $w^{*}(T,s)=(\delta\bar{g}(p^{*})/\delta p)(s)$. 
\end{theo}

\begin{proof}
We define the control function
\begin{align}
	u^{\ve}(t,z):=\begin{cases}
		u^{*}(t,s)&(t,s)\in([0,T]\times\mb{R}^{d_{s}})\backslash(E_{\ve_{1}}\times F_{\ve_{2}}),\\
		u(t,s)&(t,s)\in E_{\ve_{1}}\times F_{\ve_{2}},\\
	\end{cases}
\end{align}
where $E_{\ve_{1}}:=[t',t'+\ve_{1}]\subseteq[0,T]$, $F_{\ve_{2}}:=[s',s'+\ve_{2}]\subseteq\mb{R}^{d_{s}}$, and $^\forall u:[0,T]\times\mb{R}^{d_{s}}\to\mb{R}^{d_{u}}$. 
From the Lemma \ref{lemm: J-J MFC}, $J[u^{\ve}]-J[u^{*}]$ can be calculated as follows: 
\begin{align}
	J[u^{\ve}]-J[u^{*}]
	&=\int_{0}^{T}\left(\bar{\mcal{H}}(t,p^{\ve},u^{\ve},w^{*})-\bar{\mcal{H}}(t,p^{*},u^{*},w^{*})-\int\frac{\delta\bar{\mcal{H}}\left(t,p^{*},u^{*},w^{*}\right)}{\delta p}(s)\left(p^{\ve}(t,s)-p^{*}(t,s)\right)ds\right)dt\nonumber\\
	&\ \ \ +\bar{g}(p^{\ve})-\bar{g}(p^{*})-\int\frac{\delta\bar{g}(p^{*})}{\delta p}(s)\left(p^{\ve}(T,s)-p^{*}(T,s)\right)ds\nonumber\\
	&=\int_{0}^{T}\left(\bar{\mcal{H}}(t,p^{\ve},u^{*},w^{*})-\bar{\mcal{H}}(t,p^{*},u^{*},w^{*})-\int\frac{\delta\bar{\mcal{H}}\left(t,p^{*},u^{*},w^{*}\right)}{\delta p}(s)\left(p^{\ve}(t,s)-p^{*}(t,s)\right)ds\right)dt\nonumber\\
	&\ \ \ +\bar{g}(p^{\ve})-\bar{g}(p^{*})-\int\frac{\delta\bar{g}(p^{*})}{\delta p}(s)\left(p^{\ve}(T,s)-p^{*}(T,s)\right)ds\nonumber\\
	&\ \ \ +\int_{E_{\ve_{1}}}\int_{F_{\ve_{2}}}\left(\mcal{H}(t,s,p^{\ve},u,w^{*})-\mcal{H}(t,s,p^{\ve},u^{*},w^{*})\right)p^{\ve}(t,s)dsdt.
\end{align}
Letting $\ve_{1}\to0$ and $\ve_{2}\to0$, 
\begin{align}
	J[u^{\ve}]-J[u^{*}]
	&=\int_{0}^{T}\left(\int\frac{\delta\bar{\mcal{H}}\left(t,p^{*},u^{*},w^{*}\right)}{\delta p}(s)\left(p^{\ve}(t,s)-p^{*}(t,s)\right)ds\right.\nonumber\\
	&\ \ \ \left.-\int\frac{\delta\bar{\mcal{H}}\left(t,p^{*},u^{*},w^{*}\right)}{\delta p}(s)\left(p^{\ve}(t,s)-p^{*}(t,s)\right)ds\right)dt\nonumber\\
	&\ \ \ +\int\frac{\delta\bar{g}(p^{*})}{\delta p}(s)\left(p^{\ve}(T,s)-p^{*}(T,s)\right)ds-\int\frac{\delta\bar{g}(p^{*})}{\delta p}(s)\left(p^{\ve}(T,s)-p^{*}(T,s)\right)ds\nonumber\\
	&\ \ \ +\left(\mcal{H}(t',s',p^{*},u,w^{*})-\mcal{H}(t',s',p^{*},u^{*},w^{*})\right)p^{*}(t',s')dsdt\nonumber\\
	&=\left(\mcal{H}(t',s',p^{*},u,w^{*})-\mcal{H}(t',s',p^{*},u^{*},w^{*})\right)p^{*}(t',s')dsdt.
\end{align}
Because $u^{*}$ is the optimal control function, the following inequality is satisfied: 
\begin{align}
	0\leq J[u^{\ve}]-J[u^{*}]
	&=\left(\mcal{H}(t',s',p^{*},u,w^{*})-\mcal{H}(t',s',p^{*},u^{*},w^{*})\right)p^{*}(t',s')dsdt.
\end{align}
Therefore, (\ref{eq: optimal control of MFC}) is obtained. 
\end{proof}

\subsection{Sufficient Condition}\label{appendix: MFC-SC}
The Pontryagin's minimum principle becomes a necessary and sufficient condition if the expected Hamiltonian and the expected terminal cost function are convex as follows: 

\begin{prop}\label{prop: SC MFC}
Assume that the expected Hamiltonian $\bar{\mcal{H}}(t,p,u,w)$ is convex with respect to $p$ and $u$ and the expected terminal cost function $\bar{g}(p)$ is convex with respect to $p$. 
If the control function $u^{*}$ satisfies (\ref{eq: optimal control of MFC}), it is the optimal control function of MFSC. 
\end{prop}

\begin{proof}
We define the arbitrary control function $^\forall u: [0,T]\times\mb{R}^{d_{s}}\to\mb{R}^{d_{u}}$. 
From the Lemma \ref{lemm: J-J MFC}, $J[u]-J[u^{*}]$ is given by the following equation: 
\begin{align}
	J[u]-J[u^{*}]
	&=\int_{0}^{T}\left(\bar{\mcal{H}}(t,p,u,w^{*})-\bar{\mcal{H}}(t,p^{*},u^{*},w^{*})-\int\frac{\delta \bar{\mcal{H}}\left(t,p^{*},u^{*},w^{*}\right)}{\delta p}(s)\left(p(t,s)-p^{*}(t,s)\right)ds\right)dt\nonumber\\
	&\ \ \ +\bar{g}(p)-\bar{g}(p^{*})
	-\int\frac{\delta \bar{g}(p^{*})}{\delta p}(s)\left(p(T,s)-p^{*}(T,s)\right)ds.
\end{align}
Because $\bar{\mcal{H}}(t,p,u,w)$ is convex with respect to $p$ and $u$ and $\bar{g}(p)$ is convex with respect to $p$, the following inequalities are satisfied: 
\begin{align}
	\bar{\mcal{H}}(t,p,u,w^{*})
	&\geq\bar{\mcal{H}}(t,p^{*},u^{*},w^{*})+\int\frac{\delta\bar{\mcal{H}}(t,p^{*},u^{*},w^{*})}{\delta p}(s)(p(t,s)-p^{*}(t,s))ds\nonumber\\
	&\ \ \ +\int\left(\frac{\delta\bar{\mcal{H}}(t,p^{*},u^{*},w^{*})}{\delta u}(s)\right)^{\top}(u(t,s)-u^{*}(t,s))ds,\\
	\bar{g}(p)
	&\geq\bar{g}(p^{*})+\int\frac{\delta\bar{g}(p^{*})}{\delta p}(s)(p(T,s)-p^{*}(T,s))ds.
\end{align}
Hence, the following inequality is satisfied: 
\begin{align}
	J[u]-J[u^{*}]
	&\geq\int_{0}^{T}\mb{E}_{p^{*}(t,s)}\left[\left(\frac{\pl \mcal{H}(t,s,p^{*},u^{*},w^{*})}{\pl u}\right)^{\top}(u(t,s)-u^{*}(t,s))\right]dt. 
\end{align}
Because $u^{*}$ satisfies (\ref{eq: optimal control of MFC}), the following stationary condition is satisfied: 
\begin{align}
	\frac{\pl \mcal{H}(t,s,p^{*},u^{*},w^{*})}{\pl u}=0. 
\end{align}
Hence, the following inequality is satisfied: 
\begin{align}
	J[u]-J[u^{*}]
	&\geq0
\end{align}
Therefore, $u^{*}$ is the optimal control function.
\end{proof}

\subsection{Relationship with Bellman's Dynamic Programming Principle}\label{appendix: MFC-DPP}
From the Bellman's dynamic programming principle on the probability density function space \cite{lauriere_dynamic_2014,lauriere_dynamic_2016,pham_bellman_2018}, the optimal control function of MFSC is given by the following equation: 
\begin{align}
	u^{*}(t,s,p)=\argmin_{u}\mcal{H}\left(t,s,p,u,\frac{\delta V^{*}(t,p)}{\delta p}(s)\right). 
	\label{eq: optimal control of  MFC DPP}
\end{align}
More specifically, the optimal control function is given by $u^{*}(t,s)=u^{*}(t,s,p^{*})$, where $p^{*}(t,s)$ is the solution of the FP equation (\ref{eq: optimal FP eq MFC}). 
$V^{*}(t,p)$ is the value function on the probability density function space, which is the solution of the following Bellman equation: 
\begin{align}
	-\frac{\pl V^{*}(t,p)}{\pl t}=\mb{E}_{p(s)}\left[\mcal{H}\left(t,s,p,u^{*},\frac{\delta V^{*}(t,p)}{\delta p}(s)\right)\right],
	\label{eq: Bellman eq MFC}
\end{align}
where $V^{*}(T,p)=\mb{E}_{p(s)}\left[g(s)\right]$. 
Because the Bellman equation (\ref{eq: Bellman eq MFC}) is a functional differential equation, it cannot be solved even numerically. 
In order to resolve this problem, the previous works \cite{bensoussan_master_2015,bensoussan_interpretation_2017} converted the Bellman equation (\ref{eq: Bellman eq MFC}) into the system of HJB-FP equation (\ref{eq: optimal FP eq MFC}) and (\ref{eq: optimal HJB eq MFC}) as follows: 

\begin{prop}[\cite{bensoussan_master_2015,bensoussan_interpretation_2017}]\label{prop: DPP MFC}
We define $w^{*}(t,s)$ from $V^{*}(t,p)$ as follows: 
\begin{align}
	w^{*}(t,s):=\frac{\delta V^{*}(t,p^{*})}{\delta p}(s), 
\end{align}
where $p^{*}(t,s)$ is the solution of the FP equation (\ref{eq: optimal FP eq MFC}). 
$w^{*}(t,s)$ satisfies the HJB equation (\ref{eq: optimal HJB eq MFC}) from the Bellman equation (\ref{eq: Bellman eq MFC}). 
\end{prop}

\begin{proof}
The proof is shown in \cite{bensoussan_master_2015,bensoussan_interpretation_2017}. 
\end{proof}

This approach can be interpreted as the conversion from the Bellman's dynamic programming principle into the Pontryagin's minimum principle on the probability density function space.

\section{Proof}\label{appendix: proof}
\subsection{Proof of Lemma \ref{lemm: J-J ML-POSC}}\label{appendix: proof: lemm: J-J ML-POSC}
$J[u]-J[u']$ can be calculated as follows: 
\begin{align}
	J[u]-J[u']
	&=\mb{E}_{p(s_{0:T})}\left[\int_{0}^{T}f(t,s_{t},u_{t})dt+g(s_{T})\right]-\mb{E}_{p'(s_{0:T})}\left[\int_{0}^{T}f(t,s_{t},u_{t}')dt+g(s_{T})\right]\nonumber\\
	&=\mb{E}_{p(s_{0:T})}\left[\int_{0}^{T}\left(\mcal{H}(t,s_{t},u_{t},w')-\mcal{L}_{u_{t}}w'(t,s_{t})\right)dt+g(s_{T})\right]\nonumber\\
	&\ \ \ -\mb{E}_{p'(s_{0:T})}\left[\int_{0}^{T}\left(\mcal{H}(t,s_{t},u_{t}',w')-\mcal{L}_{u_{t}'}w'(t,s_{t})\right)dt+g(s_{T})\right]\nonumber\\
	&=\int_{0}^{T}\left(\mb{E}_{p(t,s)}\left[\mcal{H}(t,s,u,w')\right]-\mb{E}_{p'(t,s)}\left[\mcal{H}(t,s,u',w')\right]\right)dt\nonumber\\
	&\ \ \ -\int_{0}^{T}\left(\mb{E}_{p(t,s)}\left[\mcal{L}_{u}w'(t,s)\right]-\mb{E}_{p'(t,s)}\left[\mcal{L}_{u'}w'(t,s)\right]\right)dt\nonumber\\
	&\ \ \ +\mb{E}_{p(T,s)}\left[g(s)\right]-\mb{E}_{p'(T,s)}\left[g(s)\right].
\end{align}
Because $\mcal{L}_{u_{t}}$ and $\mcal{L}_{u_{t}'}$ are the conjugates of $\mcal{L}_{u_{t}}^{\dag}$ and $\mcal{L}_{u_{t}'}^{\dag}$, respectively, 
\begin{align}
	J[u]-J[u']
	&=\int_{0}^{T}\left(\mb{E}_{p(t,s)}\left[\mcal{H}(t,s,u,w')\right]-\mb{E}_{p'(t,s)}\left[\mcal{H}(t,s,u',w')\right]\right)dt\nonumber\\
	&\ \ \ -\int_{0}^{T}\int \left(\mcal{L}_{u}^{\dag}p(t,s)-\mcal{L}_{u'}^{\dag}p'(t,s)\right)w'(t,s)dsdt\nonumber\\
	&\ \ \ +\mb{E}_{p(T,s)}\left[g(s)\right]-\mb{E}_{p'(T,s)}\left[g(s)\right].
\end{align}
From the FP equation (\ref{eq: FP eq}), 
\begin{align}
	J[u]-J[u']
	&=\int_{0}^{T}\left(\mb{E}_{p(t,s)}\left[\mcal{H}(t,s,u,w')\right]-\mb{E}_{p'(t,s)}\left[\mcal{H}(t,s,u',w')\right]\right)dt\nonumber\\
	&\ \ \ -\int_{0}^{T}\int \frac{\pl\left(p(t,s)-p'(t,s)\right)}{\pl t}w'(t,s)dsdt\nonumber\\
	&\ \ \ +\mb{E}_{p(T,s)}\left[g(s)\right]-\mb{E}_{p'(T,s)}\left[g(s)\right].
\end{align}
From the integration by parts and $p(0,s)-p'(0,s)=p_{0}(s)-p_{0}(s)=0$, 
\begin{align}
	J[u]-J[u']
	&=\int_{0}^{T}\left(\mb{E}_{p(t,s)}\left[\mcal{H}(t,s,u,w')\right]-\mb{E}_{p'(t,s)}\left[\mcal{H}(t,s,u',w')\right]\right)dt\nonumber\\
	&\ \ \ +\int_{0}^{T}\int \left(p(t,s)-p'(t,s)\right)\frac{\pl w'(t,s)}{\pl t}dsdt\nonumber\\
	&\ \ \ +\mb{E}_{p(T,s)}\left[g(s)\right]-\mb{E}_{p'(T,s)}\left[g(s)\right]-\int \left(p(T,s)-p'(T,s)\right)w'(T,s)ds.
\end{align} 
From the HJB equation (\ref{eq: HJB eq}),  (\ref{eq: J-J ML-POSC}) is obtained. 

\subsection{Proof of Theorem \ref{theo: NC ML-POSC}}\label{appendix: proof: theo: NC ML-POSC}
We define the control function
\begin{align}
	u^{\ve}(t,z):=\begin{cases}
		u^{*}(t,z)&(t,z)\in([0,T]\times\mb{R}^{d_{z}})\backslash(E_{\ve_{1}}\times F_{\ve_{2}}),\\
		u(t,z)&(t,z)\in E_{\ve_{1}}\times F_{\ve_{2}},\\
	\end{cases}
\end{align}
where $E_{\ve_{1}}:=[t',t'+\ve_{1}]\subseteq[0,T]$, $F_{\ve_{2}}:=[z',z'+\ve_{2}]\subseteq\mb{R}^{d_{z}}$, and $^\forall u:[0,T]\times\mb{R}^{d_{z}}\to\mb{R}^{d_{u}}$. 
From the Lemma \ref{lemm: J-J ML-POSC}, $J[u^{\ve}]-J[u^{*}]$ can be calculated as follows: 
\begin{align}
	J[u^{\ve}]-J[u^{*}]
	&=\int_{0}^{T}\left(\mb{E}_{p^{\ve}(t,s)}\left[\mcal{H}(t,s,u^{\ve},w^{*})\right]-\mb{E}_{p^{\ve}(t,s)}\left[\mcal{H}(t,s,u^{*},w^{*})\right]\right)dt\nonumber\\
	&=\int_{E_{\ve_{1}}}\int_{F_{\ve_{2}}}\left(\mb{E}_{p_{t}^{\ve}(x|z)}\left[\mcal{H}(t,s,u,w^{*})\right]-\mb{E}_{p_{t}^{\ve}(x|z)}\left[\mcal{H}(t,s,u^{*},w^{*})\right]\right)p_{t}^{\ve}(z)dzdt.\nonumber
\end{align}
Letting $\ve_{1}\to0$ and $\ve_{2}\to0$, 
\begin{align}
	J[u^{\ve}]-J[u^{*}]
	&=\left(\mb{E}_{p_{t'}^{*}(x'|z')}\left[\mcal{H}(t',s',u,w^{*})\right]-\mb{E}_{p_{t'}^{*}(x'|z')}\left[\mcal{H}(t',s',u^{*},w^{*})\right]\right)p_{t'}^{*}(z')dzdt.\nonumber
\end{align}
Because $u^{*}$ is the optimal control function, the following inequality is satisfied: 
\begin{align}
	0\leq J[u^{\ve}]-J[u^{*}]
	&=\left(\mb{E}_{p_{t'}^{*}(x'|z')}\left[\mcal{H}(t',s',u,w^{*})\right]-\mb{E}_{p_{t'}^{*}(x'|z')}\left[\mcal{H}(t',s',u^{*},w^{*})\right]\right)p_{t'}^{*}(z')dzdt.\nonumber
\end{align}
Therefore, (\ref{eq: optimal control of ML-POSC}) is obtained. 

\subsection{Proof of Proposition \ref{prop: SC ML-POSC}}\label{appendix: proof: theo: SC ML-POSC}
We define the arbitrary control function $^\forall u: [0,T]\times\mb{R}^{d_{z}}\to\mb{R}^{d_{u}}$. 
From the Lemma \ref{lemm: J-J ML-POSC}, $J[u]-J[u^{*}]$ is given by the following equation: 
\begin{align}
	J[u]-J[u^{*}]
	&=\int_{0}^{T}\left(\mb{E}_{p(t,s)}\left[\mcal{H}(t,s,u,w^{*})\right]-\mb{E}_{p(t,s)}\left[\mcal{H}(t,s,u^{*},w^{*})\right]\right)dt.
\end{align}
Because $\bar{\mcal{H}}(t,p,u,w)$ is convex with respect to $p$ and $u$, the following inequality is satisfied: 
\begin{align}
	\mb{E}_{p(t,s)}\left[\mcal{H}(t,s,u,w^{*})\right]
	&=\bar{\mcal{H}}(t,p,u,w^{*})\nonumber\\
	&\geq\bar{\mcal{H}}(t,p^{*},u^{*},w^{*})+\int\frac{\delta\bar{\mcal{H}}(t,p^{*},u^{*},w^{*})}{\delta p}(s)(p(t,s)-p^{*}(t,s))ds\nonumber\\
	&\ \ \ +\int\left(\frac{\delta\bar{\mcal{H}}(t,p^{*},u^{*},w^{*})}{\delta u}(z)\right)^{\top}(u(t,z)-u^{*}(t,z))dz.
\end{align}
Because 
\begin{align}
	\frac{\delta\bar{\mcal{H}}(t,p^{*},u^{*},w^{*})}{\delta p}(s)
	&=\left.\frac{\delta}{\delta p}\left(\int p(s)\mcal{H}(t,s,u^{*},w^{*})ds\right)\right|_{p=p^{*}}\nonumber\\
	&=\mcal{H}(t,s,u^{*},w^{*}),\\
	\frac{\delta\bar{\mcal{H}}(t,p^{*},u^{*},w^{*})}{\delta u}(z)
	&=\left.\frac{\delta}{\delta u}\left(\int p_{t}^{*}(z)\mb{E}_{p_{t}^{*}(x|z)}\left[\mcal{H}(t,s,u,w^{*})\right]dz\right)\right|_{u=u^{*}}\nonumber\\
	&=p_{t}^{*}(z)\frac{\pl\mb{E}_{p_{t}^{*}(x|z)}\left[\mcal{H}(t,s,u^{*},w^{*})\right]}{\pl u},
\end{align}
the above inequality can be calculated as follows: 
\begin{align}
	\mb{E}_{p(t,s)}\left[\mcal{H}(t,s,u,w^{*})\right]
	&\geq\int p^{*}(t,s)\mcal{H}(t,s,u^{*},w^{*})ds+\int\mcal{H}(t,s,u^{*},w^{*})(p(t,s)-p^{*}(t,s))ds\nonumber\\
	&\ \ \ +\int p_{t}^{*}(z)\left(\frac{\pl\mb{E}_{p_{t}^{*}(x|z)}\left[\mcal{H}(t,s,u^{*},w^{*})\right]}{\pl u}\right)^{\top}(u(t,z)-u^{*}(t,z))dz\nonumber\\
	&=\mb{E}_{p(t,s)}\left[\mcal{H}(t,s,u^{*},w^{*})\right]\nonumber\\
	&\ \ \ +\mb{E}_{p_{t}^{*}(z)}\left[\left(\frac{\pl \mb{E}_{p_{t}^{*}(x|z)}\left[\mcal{H}(t,s,u^{*},w^{*})\right]}{\pl u}\right)^{\top}(u(t,z)-u^{*}(t,z))\right].
\end{align}
Hence, the following inequality is satisfied: 
\begin{align}
	J[u]-J[u^{*}]
	&\geq\int_{0}^{T}\mb{E}_{p_{t}^{*}(z)}\left[\left(\frac{\pl \mb{E}_{p_{t}^{*}(x|z)}\left[\mcal{H}(t,s,u^{*},w^{*})\right]}{\pl u}\right)^{\top}(u(t,z)-u^{*}(t,z))\right]dt. 
\end{align}
Because $u^{*}$ satisfies (\ref{eq: optimal control of ML-POSC}), the following stationary condition is satisfied: 
\begin{align}
	\frac{\pl \mb{E}_{p_{t}^{*}(x|z)}\left[\mcal{H}(t,s,u^{*},w^{*})\right]}{\pl u}=0. 
\end{align}
Hence, the following inequality is satisfied: 
\begin{align}
	J[u]-J[u^{*}]
	&\geq0
\end{align}
Therefore, $u^{*}$ is the optimal control function.

\subsection{Proof of Proposition \ref{prop: SC COSC}}\label{appendix: proof: theo: SC COSC}
We define the arbitrary control function $^\forall u: [0,T]\times\mb{R}^{d_{s}}\to\mb{R}^{d_{u}}$. 
From the Lemma \ref{lemm: J-J ML-POSC}, $J[u]-J[u^{*}]$ is given by the following equation: 
\begin{align}
	J[u]-J[u^{*}]
	&=\int_{0}^{T}\left(\mb{E}_{p(t,s)}\left[\mcal{H}(t,s,u,w^{*})\right]-\mb{E}_{p(t,s)}\left[\mcal{H}(t,s,u^{*},w^{*})\right]\right)dt.
\end{align}
From (\ref{eq: optimal control of COSC}), the following inequality is satisfied: 
\begin{align}
	J[u]-J[u^{*}]
	&\geq\int_{0}^{T}\left(\mb{E}_{p(t,s)}\left[\mcal{H}(t,s,u^{*},w^{*})\right]-\mb{E}_{p(t,s)}\left[\mcal{H}(t,s,u^{*},w^{*})\right]\right)dt=0.
\end{align}
Therefore, $u^{*}$ is the optimal control function.

\subsection{Proof of Lemma \ref{lemm: current optimal control} by the time discretized method}\label{appendix: proof: lemm: current optimal control ver 1}
$u_{t}':=\argmin_{u_{t}}J[u_{0:T-dt}]$ can be calculated as follows: 
\begin{align}
	u_{t}'
	&=\argmin_{u_{t}}J[u_{0:T-dt}]\nonumber\\
	&=\argmin_{u_{t}}\mb{E}_{p(s_{0:T};u_{0:T-dt})}\left[\int_{0}^{T}f(\tau,s_{\tau},u_{\tau})d\tau+g(s_{T})\right]\nonumber\\
	&=\argmin_{u_{t}}\mb{E}_{p(s_{t:T};u_{0:T-dt})}\left[\int_{t}^{T}f(\tau,s_{\tau},u_{\tau})d\tau+g(s_{T})\right]\nonumber\\
	&=\argmin_{u_{t}}\mb{E}_{p(s_{t:T};u_{0:T-dt})}\left[f(t,s_{t},u_{t})dt+\int_{t+dt}^{T}f(\tau,s_{\tau},u_{\tau})d\tau+g(s_{T})\right]\nonumber\\
	&=\argmin_{u_{t}}\mb{E}_{p_{t}(s_{t})}\left[f(t,s_{t},u_{t})dt+\mb{E}_{p(s_{t+dt}|s_{t};u_{t})}\left[w_{t+dt}(s_{t+dt})\right]\right].
\end{align}
where $p_{t}(s)$ is the solution of the FP equation (\ref{eq: time-discretized FP eq}), and $w_{t+dt}(s)$ is defined by 
\begin{align}
	w_{t+dt}(s):=\mb{E}_{p(s_{t+2dt:T}|s_{t+dt}=s;u_{t+dt:T-dt})}\left[\int_{t+dt}^{T}f(\tau,s_{\tau},u_{\tau})d\tau+g(s_{T})\right].
\end{align}
Although $p_{t}(s)$ and $w_{t+dt}(s)$ depend on $u_{0:t-dt}$ and $u_{t+dt:T-dt}$, respectively, they do not depend on $u_{t}$ in ML-POSC. 
From the Ito's lemma, 
\begin{align}
	u_{t}'
	&=\argmin_{u_{t}}\mb{E}_{p_{t}(s_{t})}\left[f(t,s_{t},u_{t})dt+w_{t+dt}(s_{t})+\mcal{L}_{u_{t}}w_{t+dt}(s_{t})dt\right]\nonumber\\
	&=\argmin_{u_{t}}\mb{E}_{p_{t}(s)}\left[\mcal{H}(t,s,u_{t},w_{t+dt})\right]. 
\end{align}
Because the control $u_{t}$ is the function of the memory $z$ in ML-POSC, the minimization by $u_{t}$ can be exchanged with the expectation by $p_{t}(z)$ as follows:  
\begin{align}
	u_{t}'(z)=\argmin_{u_{t}}\mb{E}_{p_{t}(x|z)}\left[\mcal{H}\left(t,s,u_{t},w_{t+dt}\right)\right]=u_{t}^{*}(z).
\end{align}
Therefore, (\ref{eq: current optimal control of ML-POSC J}) is proven. 
Finally, we prove that $w_{t}(s)$ is the solution of the HJB equation (\ref{eq: time-discretized HJB eq}). 
$w_{t}(s)$ can be calculated as follows: 
\begin{align}
	w_{t}(s)
	&=\mb{E}_{p(s_{t+dt:T}|s_{t}=s;u_{t:T-dt})}\left[\int_{t}^{T}f(\tau,s_{\tau},u_{\tau})d\tau+g(s_{T})\right]\nonumber\\
	&=f(t,s,u_{t})dt+\mb{E}_{p(s_{t+dt}|s_{t}=s;u_{t})}\left[w_{t+dt}(s_{t+dt})\right]\nonumber\\
	&=f(t,s,u_{t})dt+w_{t+dt}(s)+\mcal{L}_{u_{t}}w_{t+dt}(s)dt\nonumber\\
	&=w_{t+dt}(s)+\mcal{H}(t,s,u_{t},w_{t+dt})dt,
\end{align}
where $w_{T}(s)=g(s)$. 
Therefore, (\ref{eq: time-discretized HJB eq}) is proven. 

\subsection{Proof of Lemma \ref{lemm: current optimal control} by the similar way as the Pontyragin's minimum principle}\label{appendix: proof: lemm: current optimal control ver 2}
From the Lemma \ref{lemm: J-J ML-POSC}, the following equality is satisfied: 
\begin{align}
	&J[u_{0:t-dt},u_{t},u_{t+dt:T-dt}]-J[u_{0:t-dt},u_{t}^{*},u_{t+dt:T-dt}]\nonumber\\
	&=\left(\mb{E}_{p_{t}(s)}\left[\mcal{H}(t,s,u_{t},w_{t+dt})\right]-\mb{E}_{p_{t}(s)}\left[\mcal{H}(t,s,u_{t}^{*},w_{t+dt})\right]\right)dt\nonumber\\
	&=\mb{E}_{p_{t}(z)}\left[\mb{E}_{p_{t}(x|z)}\left[\mcal{H}(t,s,u_{t},w_{t+dt})\right]-\mb{E}_{p_{t}(x|z)}\left[\mcal{H}(t,s,u_{t}^{*},w_{t+dt})\right]\right]dt.
\end{align}
From (\ref{eq: current optimal control of ML-POSC H}), the following inequality is satisfied. 
\begin{align}
	J[u_{0:t-dt},u_{t},u_{t+dt:T-dt}]-J[u_{0:t-dt},u_{t}^{*},u_{t+dt:T-dt}]
	&\geq 0.
\end{align}
Therefore, $u_{t}^{*}$ satisfies (\ref{eq: current optimal control of ML-POSC J}).

\subsection{Proof of Lemma \ref{lemm: tight monotonicity of FBSM}}\label{appendix: proof: lemm: tight monotonicity of FBSM}
We mainly prove the inequality of the forward step (\ref{eq: tight monotonicity of forward step}).  
We can prove the inequality of the backward step (\ref{eq: tight monotonicity of backward step}) in a similar way. 
In the forward step, $u_{0:t-dt}^{k+1}$ and $u_{t+dt:T-dt}^{k}$ are given, and $u_{t}^{k+1}$ is defined by $u_{t}^{k+1}:=\argmin_{u_{t}}\mb{E}_{p_{t}^{k+1}(x|z)}\left[\mcal{H}\left(t,s,u_{t},w_{t+dt}^{k}\right)\right]$. 
In this case, from  Lemma \ref{lemm: current optimal control}, the following inequality is satisfied: 
\begin{align}
	J[u_{0:t-dt}^{k+1},u_{t}^{k},u_{t+dt:T-dt}^{k}]\geq J[u_{0:t-dt}^{k+1},u_{t}^{k+1},u_{t+dt:T-dt}^{k}]. 
\end{align}
Therefore, the inequality of the forward step (\ref{eq: tight monotonicity of forward step}) is proven.

\subsection{Proof of Theorem \ref{theo: convergence of FBSM}}\label{appendix: proof: theo: convergence of FBSM}
We mainly consider the forward step. 
The similar discussion is possible for the backward step. 
If $J[u_{0:T-dt}^{k+1}]=J[u_{0:T-dt}^{k}]$ holds, then $J[u_{0:t}^{k+1},u_{t+dt:T-dt}^{k}]=J[u_{0:t-dt}^{k+1},u_{t:T-dt}^{k}]$ holds from Lemma \ref{lemm: tight monotonicity of FBSM}. 
Because $J[u_{0}^{k+1},u_{dt:T-dt}^{k}]=J[u_{0:T-dt}^{k}]$ holds, $u_{0}^{k+1}=u_{0}^{k}$ holds. 
Then, because $J[u_{0}^{k},u_{dt}^{k+1},u_{2dt:T-dt}^{k}]=J[u_{0:T-dt}^{k}]$ holds, $u_{dt}^{k+1}=u_{dt}^{k}$ holds. 
Iterating this procedure from $t=0$ to $t=T-dt$, $u_{0:T-dt}^{k+1}=u_{0:T-dt}^{k}$ holds. 
Therefore, because the HJB equation and the FP equation depend on the same control function $u_{0:T-dt}^{k+1}=u_{0:T-dt}^{k}$, $u_{0:T-dt}^{k+1}$ satisfies the Pontryagin's minimum principle (Theorem \ref{theo: NC ML-POSC}), which is the necessary condition of the optimal control function. 

\subsection{Proof of Proposition \ref{prop: FBSM for LQG problem}}\label{appendix: proof: theo: FBSM for LQG problem}
We firstly consider the initial step. 
When the control function is initialized by (\ref{eq: FBSM LQG the initial control function}), the solution of the FP equation is given by the Gaussian distribution $p_{t}^{0}(s):=\mcal{N}(s|\mu,\Lambda^{0})$, where $\mu$ is the solution of (\ref{eq: ODE of mu}) and $\Lambda^{0}$ is the solution of $\dot{\Lambda}^{0}=\mcal{F}(\Lambda^{0},\Pi^{0})$ given $\Lambda^{0}(0)=\Lambda_{0}$. 

We then consider the backward step. 
When the solution of the FP equation is given by the Gaussian distribution $p_{t}^{k}(s):=\mcal{N}(s|\mu,\Lambda^{k})$, the solution of the HJB equation is given by the quadratic function $w_{t}^{k+1}(s)=s^{\top}\Pi^{k+1}s+(\alpha^{k+1})^{\top}s+\beta^{k+1}$, where $\Pi^{k+1}$, $\alpha^{k+1}$, and $\beta^{k+1}$ are the solutions of the following ODEs: 
\begin{align}
	-\dot{\Pi}^{k+1}&=\mcal{G}(\Lambda^{k},\Pi^{k+1}),\label{eq: ODE of Pi BS}\\
	-\dot{\alpha}^{k+1}&=(A-BR^{-1}B^{\top}\Pi^{k+1})^{\top}\alpha^{k+1}-2(I-K(\Lambda^{k}))^{\top}\Pi^{k+1} BR^{-1}B^{\top}\Pi^{k+1} (I-K(\Lambda^{k}))\mu,\label{eq: ODE of alpha BS}\\
	-\dot{\beta}^{k+1}&=\tr\left(\Pi^{k+1}\sigma\sigma^{\top}\right)-\frac{1}{4}(\alpha^{k+1})^{\top}BR^{-1}B^{\top}\alpha^{k+1}+\mu^{\top}(I-K(\Lambda^{k}))^{\top}\Pi^{k+1} BR^{-1}B^{\top}\Pi^{k+1}(I-K(\Lambda^{k}))\mu,\label{eq: ODE of beta BS}
\end{align}
where $\Pi^{k+1}(T)=P$, $\alpha^{k+1}(T)=0$, and $\beta^{k+1}(T)=0$. 

We finally consider the forward step. 
When the solution of the HJB equation is given by the quadratic function $w_{t}^{k}(s)=s^{\top}\Pi^{k}s+(\alpha^{k})^{\top}s+\beta^{k}$, the solution of the FP equation is given by the Gaussian distribution $p_{t}^{k+1}(s):=\mcal{N}(s|\mu,\Lambda^{k+1})$, where $\mu$ is the solution of (\ref{eq: ODE of mu}) and $\Lambda^{k+1}$ is the solution of $\dot{\Lambda}^{k+1}=\mcal{F}(\Lambda^{k+1},\Pi^{k})$ given $\Lambda^{k+1}(0)=\Lambda_{0}$. 
Therefore, FBSM is reduced from Algorithm \ref{alg: FBSM} to Algorithm \ref{alg: FBSM LQG} in the LQG problem. 
The details of these calculations are almost the same with \cite{tottori_memory-limited_2022}.

\section*{ACKNOWLEDGMENT}
The first author received a JSPS Research Fellowship (Grant No. 21J20436). 
This work was supported by JSPS KAKENHI (Grant No. 19H05799) and JST CREST (Grant No. JPMJCR2011).

\bibliographystyle{ieeetr}
\bibliography{220709_ML-POSC_FBSM_ref}

\begin{thebibliography}{10}

\bibitem{fox_minimum-information_2016}
R.~Fox and N.~Tishby, ``Minimum-information {LQG} control {Part} {II}:
  {Retentive} controllers,'' in {\em 2016 {IEEE} 55th {Conference} on
  {Decision} and {Control} ({CDC})}, pp.~5603--5609, Dec. 2016.

\bibitem{fox_minimum-information_2016-1}
R.~Fox and N.~Tishby, ``Minimum-information {LQG} control part {I}:
  {Memoryless} controllers,'' in {\em 2016 {IEEE} 55th {Conference} on
  {Decision} and {Control} ({CDC})}, pp.~5610--5616, Dec. 2016.

\bibitem{li_iterative_2006}
W.~Li and E.~Todorov, ``An {Iterative} {Optimal} {Control} and {Estimation}
  {Design} for {Nonlinear} {Stochastic} {System},'' in {\em Proceedings of the
  45th {IEEE} {Conference} on {Decision} and {Control}}, pp.~3242--3247, Dec.
  2006.

\bibitem{li_iterative_2007}
W.~Li and E.~Todorov, ``Iterative linearization methods for approximately
  optimal control and estimation of non-linear stochastic system,'' {\em
  International Journal of Control}, vol.~80, pp.~1439--1453, Sept. 2007.

\bibitem{nakamura_connection_2021}
K.~Nakamura and T.~J. Kobayashi, ``Connection between the {Bacterial}
  {Chemotactic} {Network} and {Optimal} {Filtering},'' {\em Physical Review
  Letters}, vol.~126, p.~128102, Mar. 2021.

\bibitem{nakamura_optimal_2022}
K.~Nakamura and T.~J. Kobayashi, ``Optimal sensing and control of
  run-and-tumble chemotaxis,'' {\em Physical Review Research}, vol.~4,
  p.~013120, Feb. 2022.

\bibitem{pezzotta_chemotaxis_2018}
A.~Pezzotta, M.~Adorisio, and A.~Celani, ``Chemotaxis emerges as the optimal
  solution to cooperative search games,'' {\em Physical Review E}, vol.~98,
  p.~042401, Oct. 2018.

\bibitem{borra_optimal_2021}
F.~Borra, M.~Cencini, and A.~Celani, ``Optimal collision avoidance in swarms of
  active {Brownian} particles,'' {\em Journal of Statistical Mechanics: Theory
  and Experiment}, vol.~2021, p.~083401, Aug. 2021.

\bibitem{tottori_memory-limited_2022}
T.~Tottori and T.~J. Kobayashi, ``Memory-{Limited} {Partially} {Observable}
  {Stochastic} {Control} and {Its} {Mean}-{Field} {Control} {Approach},'' {\em
  Entropy}, vol.~24, p.~1599, Nov. 2022.

\bibitem{yong_stochastic_1999}
J.~Yong and X.~Y. Zhou, {\em Stochastic {Controls}}.
\newblock New York, NY: Springer New York, 1999.

\bibitem{kushner_numerical_1992}
H.~J. Kushner and P.~G. Dupuis, {\em Numerical {Methods} for {Stochastic}
  {Control} {Problems} in {Continuous} {Time}}.
\newblock New York, NY: Springer US, 1992.

\bibitem{fleming_controlled_2006}
W.~H. Fleming and H.~M. Soner, {\em Controlled {Markov} processes and viscosity
  solutions}.
\newblock No.~25 in Applications of mathematics, New York: Springer, 2nd
  ed~ed., 2006.

\bibitem{puterman_markov_2014}
M.~L. Puterman, {\em Markov {Decision} {Processes}: {Discrete} {Stochastic}
  {Dynamic} {Programming}}.
\newblock Wiley-Interscience, Aug. 2014.

\bibitem{vinter_optimal_2010}
R.~Vinter, {\em Optimal {Control}}.
\newblock Boston: Birkh^^c3^^a4user Boston, 2010.

\bibitem{lewis_optimal_2012}
F.~L. Lewis, D.~Vrabie, and V.~L. Syrmos, {\em Optimal {Control}}.
\newblock John Wiley \& Sons, Mar. 2012.

\bibitem{aschepkov_optimal_2016}
L.~T. Aschepkov, D.~V. Dolgy, T.~Kim, and R.~P. Agarwal, {\em Optimal
  {Control}}.
\newblock Cham: Springer International Publishing, 2016.

\bibitem{bensoussan_mean_2013}
A.~Bensoussan, J.~Frehse, and P.~Yam, {\em Mean {Field} {Games} and {Mean}
  {Field} {Type} {Control} {Theory}}.
\newblock Springer {Briefs} in {Mathematics}, New York, NY: Springer New York,
  2013.

\bibitem{carmona_probabilistic_2018}
R.~Carmona and F.~Delarue, {\em Probabilistic {Theory} of {Mean} {Field}
  {Games} with {Applications} {I}}.
\newblock No.~volume 83 in Probability theory and stochastic modelling, Cham:
  Springer Nature, 2018.

\bibitem{carmona_probabilistic_2018-1}
R.~Carmona and F.~Delarue, {\em Probabilistic {Theory} of {Mean} {Field}
  {Games} with {Applications} {II}}, vol.~84 of {\em Probability {Theory} and
  {Stochastic} {Modelling}}.
\newblock Cham: Springer International Publishing, 2018.

\bibitem{crisan_master_2014}
R.~Carmona and F.~Delarue, ``The {Master} {Equation} for {Large} {Population}
  {Equilibriums},'' in {\em Stochastic {Analysis} and {Applications} 2014}
  (D.~Crisan, B.~Hambly, and T.~Zariphopoulou, eds.), vol.~100, pp.~77--128,
  Cham: Springer International Publishing, 2014.
\newblock Series Title: Springer Proceedings in Mathematics \& Statistics.

\bibitem{bensoussan_master_2015}
A.~Bensoussan, J.~Frehse, and S.~C.~P. Yam, ``The {Master} equation in mean
  field theory,'' {\em Journal de Math^^c3^^a9matiques Pures et
  Appliqu^^c3^^a9es}, vol.~103, pp.~1441--1474, June 2015.

\bibitem{bensoussan_interpretation_2017}
A.~Bensoussan, J.~Frehse, and S.~C.~P. Yam, ``On the interpretation of the
  {Master} {Equation},'' {\em Stochastic Processes and their Applications},
  vol.~127, pp.~2093--2137, July 2017.

\bibitem{krylov_method_1963}
I.~Krylov and F.~Chernous'ko, ``On a method of successive approximations for
  the solution of problems of optimal control,'' {\em USSR Computational
  Mathematics and Mathematical Physics}, vol.~2, pp.~1371--1382, Jan. 1963.

\bibitem{mitter_successive_1966}
S.~K. Mitter, ``Successive approximation methods for the solution of optimal
  control problems,'' {\em Automatica}, vol.~3, pp.~135--149, Jan. 1966.

\bibitem{chernousko_method_1982}
F.~L. Chernousko and A.~A. Lyubushin, ``Method of successive approximations for
  solution of optimal control problems,'' {\em Optimal Control Applications and
  Methods}, vol.~3, no.~2, pp.~101--114, 1982.

\bibitem{lenhart_optimal_2007}
S.~Lenhart and J.~T. Workman, {\em Optimal {Control} {Applied} to {Biological}
  {Models}}.
\newblock New York: Chapman and Hall/CRC, May 2007.

\bibitem{sharp_implementation_2021}
J.~A. Sharp, K.~Burrage, and M.~J. Simpson, ``Implementation and acceleration
  of optimal control for systems biology,'' {\em Journal of The Royal Society
  Interface}, vol.~18, no.~181, p.~20210241, 2021.
\newblock Publisher: Royal Society.

\bibitem{hackbusch_numerical_1978}
W.~Hackbusch, ``A numerical method for solving parabolic equations with
  opposite orientations,'' {\em Computing}, vol.~20, pp.~229--240, Sept. 1978.

\bibitem{mcasey_convergence_2012}
M.~McAsey, L.~Mou, and W.~Han, ``Convergence of the forward-backward sweep
  method in optimal control,'' {\em Computational Optimization and
  Applications}, vol.~53, pp.~207--226, Sept. 2012.

\bibitem{carlini_semi-lagrangian_2013}
E.~Carlini and F.~J. Silva, ``Semi-{Lagrangian} schemes for mean field game
  models,'' in {\em 52nd {IEEE} {Conference} on {Decision} and {Control}},
  pp.~3115--3120, Dec. 2013.
\newblock ISSN: 0191-2216.

\bibitem{carlini_fully_2014}
E.~Carlini and F.~J. Silva, ``A {Fully} {Discrete} {Semi}-{Lagrangian} {Scheme}
  for a {First} {Order} {Mean} {Field} {Game} {Problem},'' {\em SIAM Journal on
  Numerical Analysis}, vol.~52, pp.~45--67, Jan. 2014.
\newblock Publisher: Society for Industrial and Applied Mathematics.

\bibitem{carlini_semi-lagrangian_2015}
E.~Carlini and F.~J. Silva, ``A semi-{Lagrangian} scheme for a degenerate
  second order mean field game system,'' {\em Discrete \& Continuous Dynamical
  Systems}, vol.~35, no.~9, p.~4269, 2015.

\bibitem{lauriere_numerical_2021}
M.~Lauriere, ``Numerical {Methods} for {Mean} {Field} {Games} and {Mean}
  {Field} {Type} {Control},'' June 2021.
\newblock arXiv:2106.06231 [cs, math].

\bibitem{bensoussan_estimation_2018}
A.~Bensoussan, {\em Estimation and {Control} of {Dynamical} {Systems}}, vol.~48
  of {\em Interdisciplinary {Applied} {Mathematics}}.
\newblock Cham: Springer International Publishing, 2018.

\bibitem{li_maximum_2018}
Q.~Li, L.~Chen, C.~Tai, and W.~E, ``Maximum {Principle} {Based} {Algorithms}
  for {Deep} {Learning},'' {\em Journal of Machine Learning Research}, vol.~18,
  no.~165, pp.~1--29, 2018.

\bibitem{liu_symplectic_2021}
X.~Liu and J.~Frank, ``Symplectic {Runge}^^e2^^80^^93{Kutta} discretization of
  a regularized forward^^e2^^80^^93backward sweep iteration for optimal control
  problems,'' {\em Journal of Computational and Applied Mathematics}, vol.~383,
  p.~113133, Feb. 2021.

\bibitem{bellman_dynamic_1957}
R.~E. Bellman, {\em Dynamic {Programming}}.
\newblock Princeton, USA: Princeton University Press, 1957.

\bibitem{howard_dynamic_1960}
R.~A. Howard, {\em Dynamic programming and {Markov} processes}.
\newblock Dynamic programming and {Markov} processes, Oxford, England: John
  Wiley, 1960.
\newblock Pages: viii, 136.

\bibitem{kappen_linear_2005}
H.~J. Kappen, ``Linear {Theory} for {Control} of {Nonlinear} {Stochastic}
  {Systems},'' {\em Physical Review Letters}, vol.~95, p.~200201, Nov. 2005.

\bibitem{kappen_path_2005}
H.~J. Kappen, ``Path integrals and symmetry breaking for optimal control
  theory,'' {\em Journal of Statistical Mechanics: Theory and Experiment},
  vol.~2005, pp.~P11011--P11011, Nov. 2005.

\bibitem{satoh_iterative_2017}
S.~Satoh, H.~J. Kappen, and M.~Saeki, ``An {Iterative} {Method} for {Nonlinear}
  {Stochastic} {Optimal} {Control} {Based} on {Path} {Integrals},'' {\em IEEE
  Transactions on Automatic Control}, vol.~62, pp.~262--276, Jan. 2017.
\newblock Conference Name: IEEE Transactions on Automatic Control.

\bibitem{cacace_policy_2021}
S.~Cacace, F.~Camilli, and A.~Goffi, ``A policy iteration method for {Mean}
  {Field} {Games},'' {\em arXiv:2007.04818 [math]}, July 2021.
\newblock arXiv: 2007.04818.

\bibitem{lauriere_policy_2021}
M.~Lauri^^c3^^a8re, J.~Song, and Q.~Tang, ``Policy iteration method for
  time-dependent {Mean} {Field} {Games} systems with non-separable
  {Hamiltonians},'' {\em arXiv:2110.02552 [cs, math]}, Oct. 2021.
\newblock arXiv: 2110.02552.

\bibitem{camilli_rates_2022}
F.~Camilli and Q.~Tang, ``Rates of convergence for the policy iteration method
  for {Mean} {Field} {Games} systems,'' Mar. 2022.
\newblock arXiv:2108.00755 [math].

\bibitem{ruthotto_machine_2020}
L.~Ruthotto, S.~J. Osher, W.~Li, L.~Nurbekyan, and S.~W. Fung, ``A machine
  learning framework for solving high-dimensional mean field game and mean
  field control problems,'' {\em Proceedings of the National Academy of
  Sciences}, vol.~117, pp.~9183--9193, Apr. 2020.

\bibitem{lin_alternating_2021}
A.~T. Lin, S.~W. Fung, W.~Li, L.~Nurbekyan, and S.~J. Osher, ``Alternating the
  population and control neural networks to solve high-dimensional stochastic
  mean-field games,'' {\em Proceedings of the National Academy of Sciences},
  vol.~118, Aug. 2021.

\bibitem{lauriere_dynamic_2014}
M.~Lauri^^c3^^a8re and O.~Pironneau, ``Dynamic programming for mean-field type
  control,'' {\em Comptes Rendus Mathematique}, vol.~352, pp.~707--713, Sept.
  2014.

\bibitem{lauriere_dynamic_2016}
M.~Lauri^^c3^^a8re and O.~Pironneau, ``Dynamic {Programming} for {Mean}-{Field}
  {Type} {Control},'' {\em Journal of Optimization Theory and Applications},
  vol.~169, pp.~902--924, June 2016.

\bibitem{pham_bellman_2018}
H.~Pham and X.~Wei, ``Bellman equation and viscosity solutions for mean-field
  stochastic control problem,'' {\em ESAIM: Control, Optimisation and Calculus
  of Variations}, vol.~24, pp.~437--461, Jan. 2018.

\end{thebibliography}
\end{document}